\makeatletter
\let\ORIlabel\label
\let\ORIrefstepcounter\refstepcounter
\AddToHook{package/hyperref/before}{
  \let\label\ORIlabel 
  \let\refstepcounter\ORIrefstepcounter
}
\makeatother
\documentclass[review,hidelinks,onefignum,onetabnum]{siamart220329}
\usepackage{amssymb}

\usepackage{lipsum}
\usepackage{amsfonts}
\usepackage{graphicx}
\usepackage{epstopdf}
\usepackage{algorithmic}
\usepackage{booktabs}
\ifpdf
  \DeclareGraphicsExtensions{.eps,.pdf,.png,.jpg}
\else
  \DeclareGraphicsExtensions{.eps}
\fi

\newsiamremark{remark}{Remark}
\newsiamremark{hypothesis}{Hypothesis}
\crefname{hypothesis}{Hypothesis}{Hypotheses}
\newsiamthm{claim}{Claim}
\newsiamthm{exmp}{Example}
\crefname{exmp}{Example}{Examples}
\newsiamthm{probl}{Problem}
\crefname{probl}{Problem}{Problems}

\headers{A description of the diversity index polytope}{M. Frohn and K. Manson}

\title{A minimal compact description of the diversity index polytope \thanks{
\funding{The first author acknowledges support from the European Union’s Horizon 2020 research and innovation programme under the Marie Skłodowska-Curie grant agreement no. 101034253, and by the NWO Gravitation project NETWORKS under grant no. 024.002.003, and by grant OCENW.M.21.306 from the Dutch Research Council (NWO). The second author acknowledges support from the New Zealand Marsden Fund (MFP-UOC2005).}}}

\author{Martin Frohn\thanks{Maastricht University, Department of Advanced Computing Sciences, Paul Henri Spaaklaan 1, 6229 EN Maastricht, The Netherlands (\email{martin.frohn@maastrichtuniversity.nl}). Corresponding author.}
\and Kerry Manson\thanks{University of Canterbury, School of Mathematics and Statistics, Private Bag 4800, Christchurch 8140, New Zealand
  (\email{kerry.manson@pg.canterbury.ac.nz}).}
}

\usepackage{amsopn}

\newcommand{\CorrB}[1]{{\color{black} #1}}

\begin{document}

\maketitle

\begin{abstract}
A phylogenetic tree is an edge-weighted binary tree, with leaves labelled by a collection of species, that
represents the evolutionary relationships between those species. For such a tree, a phylogenetic diversity index is a function that apportions the biodiversity of the collection across its constituent species. The diversity index polytope is the convex hull of the images of phylogenetic diversity indices. We study the combinatorics of phylogenetic diversity indices to provide a minimal compact description of the diversity index polytope. Furthermore, we discuss extensions of the polytope to expand the study of biodiversity measurement.
\end{abstract}



\section{Introduction}\label{intro}
The study of biodiversity has played a prominent role in conservation biology for centuries~\cite{myers89,Singh02,spicer13}. The current threat of extinction in many populations~\cite{iucn06} has elevated the search for prioritization protocols to protect species with high levels of evolutionary distinctiveness, genetic information or community differentiation~\cite{wright90,crozier97,hartmann06,tucker17}. Quantitative, phylogeny-based methods can inform such priorities if they are accurate and robust. Without these qualities, prioritization protocols run the risk of being unhelpful or being subject to large changes when extinctions occur or when further resolution is added to a phylogeny. 
Such views have spurred recent research into the mathematical properties of existing measures~\cite{semple23, fischer23, DIopt1, manson24, moulton2024phylogenetic}. 
Another motivator is that these methods have been used as part of real-world conservation decision making processes for determining funding allocations, such as those undertaken by the EDGE of Existence programme~\cite{EDGEofExistence_2017}.
We consider two types, namely phylogenetic diversity~\cite{faith92} and phylogenetic diversity indices~\cite{DIopt1}. The former measures the evolutionary history shared by species. The latter derive from a desire to individualize phylogenetic diversity by seeking to share out or attribute the biodiversity represented by a phylogenetic tree to the individual species it contains. However, different diversity indices can lead to quite distinct priorities and conservation strategies, making the choice of index important. The great expense and long-term nature of conservation projects both reinforce the need to make the best possible decisions.

One peculiarity of this field is that various quantitative methods appear on an ad hoc basis, as some fairly intuitive solution to a particular biological question about a particular dataset. 
The tendency for new methods to be invented, ahead of using existing ones, has lead to a so-called `jungle of indices' for biologists to choose from. 
Some have argued that the proliferation hinders rather than helps conservation management, understanding and evaluation \cite{ricotta2005through}.
Resolving this `jungle' has become a topic of interest for those interested in phylogenetic trees from a theoretical standpoint.
For this reason we advocate applying a `properties first' strategy to the family of phylogenetic diversity indices. This parallels an approach previously used to evaluate abundance-based diversity measures, such as Hill numbers~\cite{daly2018ecological}.

In this article, we follow this strategy by establishing desirable properties of quantitative phylogeny-based methods in the form of diversity indices first. Next, we show how the combinatorial properties of diversity indices can be characterized by a polytope of polynomial-size in the real numbers. This new insight makes the study of biodiversity admissible to techniques from mathematical programming, establishing a first connection (to the best of our knowledge) between polyhedral combinatorics and the study of diversity indices.
This shifts the focus on to the strength or importance of the desirable properties, while remaining agnostic about the particular methods that arise.
We expect that this approach provides a foundation for the development of new solutions in conversation biology, from which practitioners will be able to select measures that better fit the properties they require.

We begin by defining phylogenetic trees: given a set of $n$ species $X=\{x_1,\dots,x_n\}$, a \emph{rooted $X$-tree} is a connected acyclic digraph $T=(V,E)$ with leafset $X$ and a vertex $\rho\in V$ with in-degree~$0$ and out-degree~$2$ such that there exists a (unique) directed path from $\rho$ to every leaf of $T$. We call $T$ \emph{binary} if every vertex in $V\setminus(X\cup\{\rho\})$ has in-degree~$1$ and out-degree~$2$. We shall restrict our attention to binary rooted $X$-trees throughout, examples of which are shown in Figure~\ref{fig::df1}. We call $X$ the set of \emph{taxa} and $\rho$ the \emph{root} of $T$. For $Y\subseteq X$, let $T_Y=(V_Y,E_Y)$ denote the spanning subtree of $T$ rooted at $\rho$ with leafset $Y$ and let $\ell :E\to\mathbb{R}_{\geq 0}$ be an edge length function. For a binary rooted $X$-tree $T$, we call the tuple $(T,\ell)$ a \emph{rooted phylogenetic $X$-tree}. We note that the exclusion of polytomies, i.e., non-binary branching in rooted phylogenetic $X$-trees, simplifies the measurement of biodiversity~\cite{swenson2018phylogenetic}. However, our combinatorial study of diversity indices will not be affected by the binary assumption because any polytomy can be artificially resolved by introducing edges of length zero. These resolutions do not alter diversity indices as defined in the next paragraph as long as the symmetries of induced subtrees are preserved. 

To introduce a measure of biodiversity which apportions the biodiversity of $X$ across a subset of species, we formalize the similarities of paths from the root to pairs of distinct taxa. To this end, for $v\in V$ and $e\in E$, let $P(v)$ and $P(e)$ denote the directed path from root $\rho$ to (and including) vertex $v$ and edge $e$, respectively, and let $|P(v)|$ and $|P(e)|$ denote the number of edges in $P(v)$ and $P(e)$, respectively. Moreover, for $v\in V$, $v\neq\rho$, we denote $T[v]$ as the maximal subtree of $T$, rooted at $v$, obtained by deleting the edge $(u,v)\in E$. We call $T[v]$ a \emph{pendant subtree} of $T$ and denote $X[v]\subseteq X$ as the leafset of $T[v]$. For a pendant subtree $H$ of $T$, we call $T$ \emph{$H$-balanced} if there exists a maximal integer $p\geq 1$, $i(1),\dots,i(p)\in\{1,\dots,n\}$ and a collection $\left\{T\left[v_{i(1)}\right],\dots,T\left[v_{i(p)}\right]\right\}$ of pendant subtrees of $T$ such that
\begin{enumerate}
\item $\left|P\left(v_{i(1)}\right)\right|=\left|P\left(v_{i(2)}\right)\right|=\dots =\left|P\left(v_{i(p)}\right)\right|$,
\item for $q\in\{1,\dots,p\}$, pendant subtrees $T\left[v_{i(q)}\right]$ and $H$ are \emph{isomorphic}, i.e., there exists a graph isomorphism $\phi_{i(q)}$ from the vertices of $T\left[v_{i(q)}\right]$ to the vertices of $H$ which is the identity map when restricted to vertices $X$.
\end{enumerate}
If $T$ is $H$-balanced for the singleton graph $H$, then we simply say that $T$ is \emph{balanced}. For example, the pendant subtree $T'[v_4]$ on the right in Figure~\ref{fig::df1} is $H$-balanced for the rooted subtree $H$ shown in the same figure. The subtrees $T'[v_1]$ and $T'[v_6]$ are balanced. 
\begin{figure}[t]
\centering
\includegraphics[scale=0.425]{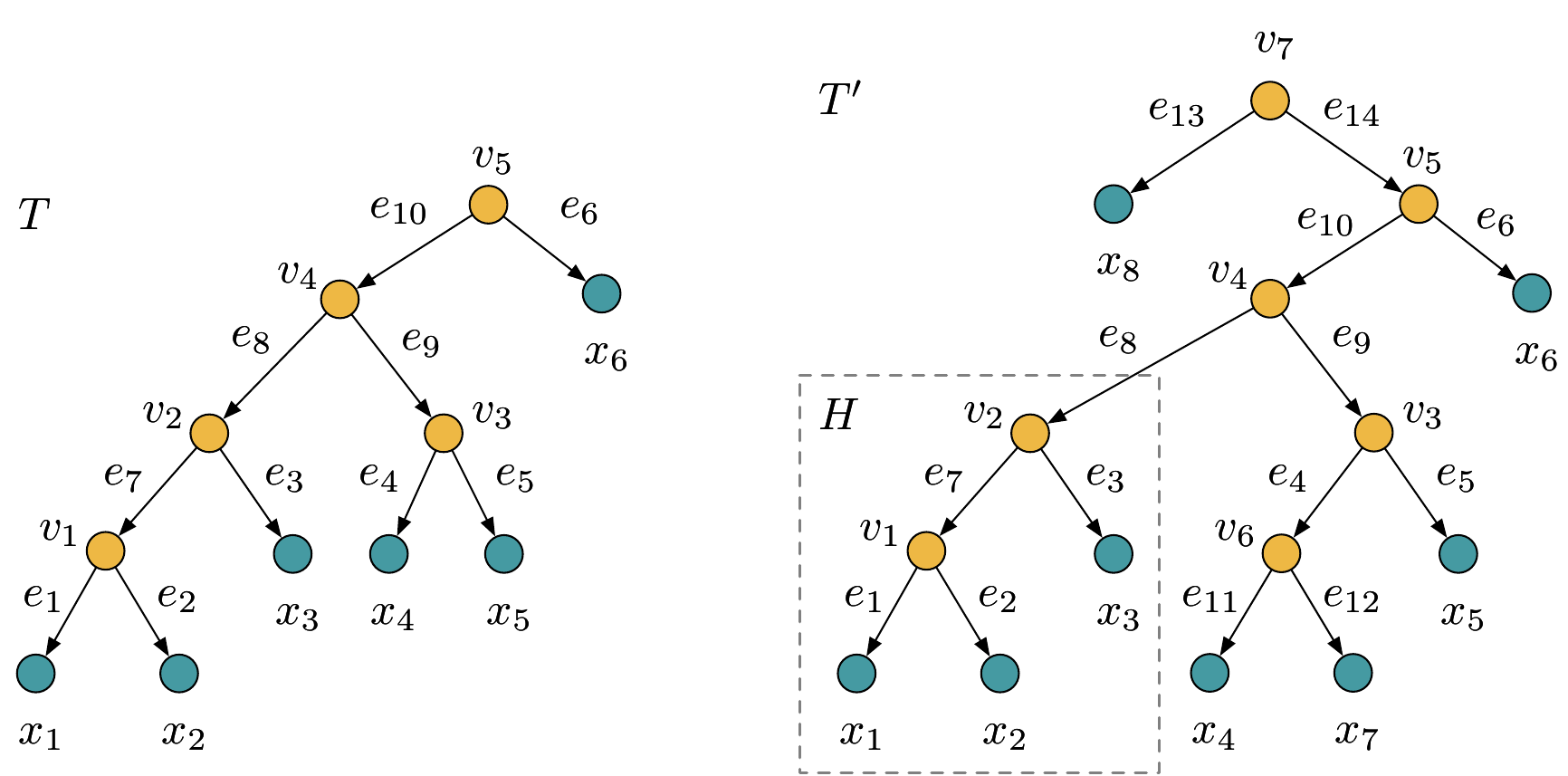}
\caption{On the left (right) a rooted $X$-tree $T$ ($X'$-tree $T'$) for $X=\{x_1,x_2,\dots,x_6\}$ ($X'=\{x_1,x_2,\dots,x_8\}$) with edges $\{e_1,e_2,\dots,e_{10}\}$ ($\{e_1,e_2,\dots,e_{14}\}$), internal vertices $\{v_1,v_2,\dots,v_5\}$ ($\{v_1,v_2,\dots,v_7\}$) and root $v_5$ ($v_7$). $H$ is a pendant subtree of $T'$.}\label{fig::df1}
\end{figure}
Observe that the notion of $H$-balanced subtrees captures local subtree structures $H$ whose replacement by singleton vertices in $T$ yields a balanced subtree of $T$ with $p$ leaves. For example, on the right in Figure~\ref{fig::df1}, \CorrB{replacing the pendant subtrees $T'[v_2]$ and $T'[v_3]$ isomorphic to $H$ by their root yields the balanced subtree rooted in $v_4$ with edge set $\{e_8,e_9\}$. Here, the replacement of two pendant subtrees by their roots is maximal for the choice of $H$.} Then, for a rooted phylogenetic $X$-tree $\hat{T}=(T,\ell)$, we call the linear function $\varphi_{\hat{T}}:X\to\mathbb{R}_{\geq 0}^{n}$ defined by
\begin{align*}
\varphi_{\hat{T}}(x)=\sum_{e\in E}\gamma(x,e)\ell(e)
\end{align*}
a \emph{diversity index} of $\hat{T}$ where, for $x\in X$, $e\in E$, the $\gamma(x,e)$ are non-negative real numbers under the following constraints:
\begin{flalign}
&\text{Convexity conditions:}\label{div::1}\\
&~\sum_{x\in X}\gamma(x,e)=1~~~~~~~~~~~~~~~~~~\,\forall\,e\in E,\nonumber\\
&\text{Descent conditions:}\label{div::2}\\
&\,~~~~~~~\gamma(x,e)=0~~~~~~~~~~~~~~~~~~\,\forall\,x\in X,e\in E\setminus P(x),\nonumber\\
&\text{Neutrality conditions:}\label{div::3}\\
&\begin{aligned}\begin{cases}
\gamma(x^1,e_1)&=\gamma(x^2,e_2)~~~~~~~~~\forall\,v_1,v_2\in V,\,x^i\in X[v_i],\,e_i=(u_i,v_i)\in E\\
&~~~~~~~~~~~~~~~~~~~~~~~~~~~\text{such that }T[v_1]\cong T[v_2],~\phi_1(x^1)=\phi_2(x^2),\\
\gamma(x^1,e)&=\gamma(x^2,e)~~~~~~~~~~\,\forall\,e=(u,v)\in E,\,x^1,x^2\in X[v],\,T[v]~\text{balanced}.
\end{cases}\end{aligned}\nonumber
\end{flalign}
For $i\in\{1,\dots,n\}$, we call $\varphi_{\hat{T}}(x_i)$ the \emph{diversity index score} of taxon~$x_i$. 
\CorrB{\begin{definition}
We call the union of images of all diversity indices for taxa $X$ the \emph{diversity index polytope} and denote it by $\mathcal{P}(T,\ell)$.
\end{definition}}
\CorrB{Observe that $\mathcal{P}(T,\ell)$ is a subset of $\mathbb{R}_{\geq 0}^n$ and} a polytope because the domain $X$ of diversity indices $\varphi_{\hat{T}}$ is bounded, conditions~\eqref{div::1} and~\eqref{div::2} are invariant under convex combinations and conditions~\eqref{div::3} are agnostic with respect to the choice of values for numbers $\gamma(x,e)$. In addition, conditions~\eqref{div::3} ensure that $H$-balanced pendant subtrees are treated by any diversity index like balanced pendant subtrees. The following example illustrates the restriction to the affine subspace defined by conditions~\eqref{div::1} to~\eqref{div::3}:
\begin{exmp}
\label{ex::def1}
Consider the rooted $X$-tree $T$ on the left in Figure~\ref{fig::df1}, any edge length function $\ell$ on $E(T)$ and any diversity index $\varphi_{(T,\ell)}$. Then, we have $\gamma(x_i,e_i)=1$ and $\gamma(x_j,e_i)=0$ for all $i,j\in\{1,\dots,6\},i\neq j$ due to the convexity and descent condition. Moreover, $x_1,x_2\in X[v_1]$, $x_4,x_5\in X[v_3]$ and $T[v_1],T[v_3]$ are balanced, i.e., $$\gamma(x_1,e_7)=\gamma(x_2,e_7)=\frac{1}{2}~~~\text{and}~~~\gamma(x_4,e_9)=\gamma(x_5,e_9)=\frac{1}{2}$$ by all three conditions~\eqref{div::1},~\eqref{div::2} and~\eqref{div::3} together. In this case, we say that the pairs $(x_1,x_2)$ and $(x_4,x_5)$ each satisfy a shared neutrality condition.
\end{exmp}
While $\mathcal{P}(T,\ell)$ contains vastly different diversity index scores, any particular choice of a diversity index can be informative under additional assumptions. For example, choosing numbers $\gamma(x,e)$, for $x\in X$, $e\in E$, such that
\begin{align*}
\gamma(x,(u,v))&=\begin{cases}\frac{1}{|X[v]|}~&\text{if }x\in X,\,(u,v)\in P(x),\\
0~&\text{otherwise.}
\end{cases}
\end{align*}
defines the \emph{Fair Proportion Index}~\cite{fischer23}. This index is defined based on the assumption that all species with a common ancestor $v$ should be treated uniformly. Observe that, for $(u,v)\in E$, the convexity condition holds:
\begin{align*}
\sum_{x\in X}\gamma(x,(u,v))=\sum_{x\in X,\,(u,v)\in P(x)}\frac{1}{|X[v]|}=\sum_{x\in X[v]}\frac{1}{|X[v]|}=1.
\end{align*}
Moreover, the descent conditions hold by definition. In addition, for $i\in\{1,2\}$, $v_i\in V$, $x^i\in X[v_i]$, $e_i=(u_i,v_i)\in E$ such that $T[v_1]\cong T[v_2]$ we know that $|X[v_1]|=|X[v_2]|$, i.e.,
\begin{align*}
\gamma(x^1,e_1)=\frac{1}{|X[v_1]|}=\frac{1}{|X[v_2]|}=\gamma(x^2,e_2)
\end{align*}
and for $e=(u,v)\in E$, $x^1,x^2\in X[v]$ we have
\begin{align*}
\gamma(x^1,e)=\frac{1}{|X[v]|}=\gamma(x^2,e).
\end{align*}
Thus, the neutrality conditions hold and therefore the Fair Proportion Index is a diversity index.

Our primary motivation to study the diversity index polytope is its usefulness for conservation projects. However, since we take a purely combinatorial approach, extensions of the diversity index polytope by adding further desired properties to the established conditions~\eqref{div::1} to~\eqref{div::3} are of interest, too. We mention this aspect briefly in the last section and do not study it in this article but point out that the analysis of additional diversity index properties can make use of the results in this article to derive more complex polytopes based on diversity indices. We will see that the diversity index polytope does not have a large number of facets because of the independence conditions~\eqref{div::2} on subtree structures. By introducing the notion of a \emph{diversity ranking}~\cite{fischer23}, i.e., a partial order on the taxa to determine which species contribute more to the biodiversity of a population (in the face of extinction events), more dependencies between the vertices of $\mathcal{P}(T,\ell)$ arise. Thus, this article provides the necessary ground work to pose questions of interest to the broader operations research community.

Before we begin our study of the polytope $\mathcal{P}(T,\ell)$ we familiarize the reader with a classic approach from polyhedral combinatorics which we will follow: a minimal compact description of $\mathcal{P}(T,\ell)$ is a linear inequality-system such that a solution to the system encodes the diversity index scores of a particular diversity index if and only if the solution is feasible. This means, from a geometric point of view we want to construct a linear inequality-system which corresponds to the description of $\mathcal{P}(T,\ell)$ as the intersection of halfspaces. If an inequality in this system defines a face of $\mathcal{P}(T,\ell)$ of maximal dimension, it is called facet-defining. Hence, after constructing the system we prove that we have included all facet-defining inequalities for $\mathcal{P}(T,\ell)$ in the system. The key insight to get to this main result will be Proposition~\ref{prop::basis} by providing a basis in a vector space embedding of $\mathcal{P}(T,\ell)$, allowing us to pairwise compare all linear independence relations between vectors of diversity index scores. Intuitively, the basis vectors we will construct look as follows: consider the phylogenetic $X$-tree on the left in Figure~\ref{fig::df1}. Choose any taxon, say $x_1$, and maximize its diversity index score $s_1$, i.e., $s_1=\ell(e_1)+(\ell(e_7)+\ell(e_8)+\ell(e_{10}))/2$. Then, consider the forest we obtain by removing path $P(x_1)$ and recursively maximize the diversity index score of taxa in the connected components of the forest without violating any of the conditions~\eqref{div::1} to~\eqref{div::3}. The resulting vector of diversity index scores will form one of our desired basis vectors for $\mathcal{P}(T,\ell)$. However, we have to address two obstacles first to eventually arrive at a basis construction. On the one hand, diversity index scores are not a local property of a phylogenetic $X$-tree because they can be defined by large weighted sums of edge lengths offering many degrees of freedom in the coefficients $\gamma(x,e)$. We will localize the calculation of diversity index scores by enforcing consistency conditions~\cite{DIopt1} during our analysis. On the other hand, the symmetries introduced by neutrality conditions~\eqref{div::3} can lower the dimension of $\mathcal{P}(T,\ell)$ significantly in any vector space embedding. Hence, we parameterize them by the notion of a set of independent edges which we introduce in the next section to provide explicit formulas for the basis vector calculation we hinted at.

In Section~\ref{sec2} we introduce some notation and background information on phylogenetic diversity indices that will prove useful to describe $\mathcal{P}(T,\ell)$. In Section~\ref{sec4}, we investigate the combinatorial properties of diversity indices to prove a minimal compact description of $\mathcal{P}(T,\ell)$. In Section~\ref{sec5} we discuss extensions of $\mathcal{P}(T,\ell)$ and further research directions.

\section{Notation and background}\label{sec2}
Let $T$ be a rooted $X$-tree. We write $V_{\text{int}}(T)$, $V(T)$ and $E(T)$ to refer to the set of interior vertices, all vertices and all edges of $T$, respectively. 
For an edge length function $\ell :E(T)\to\mathbb{R}_{\geq 0}$, to simplify our notation, we denote any restriction of $\ell$ to subsets of $E(T)$ by $\ell$, too. For $s\in \mathcal{P}(T,\ell)$, we can write $s=\Gamma_Tl$ for a suitable matrix $\Gamma_T\in\mathbb{R}^{n\times|E(T)|}$ of real numbers $\gamma(x,e)$ and a vector $l\in\mathbb{R}^{|E(T)|}$ of edge lengths. Moreover, $\Gamma_T(e)$ denotes the column vector of $\Gamma_T$ corresponding to edge $e\in E(T)$. For $i\in\{1,\dots,n\}$, we also refer to the diversity index score $s_i$ as an \emph{allocation of edge lengths} to taxon $x_i$. Hence, a diversity index can be viewed as an allocation of the total edge lengths to the leaves. Our approach to characterize $\mathcal{P}(T,\ell)$ will involve maximizing and minimizing these scores for either individual leaves or for sets of leaves.

Then, to localize the allocation of edge lengths, Manson and Steel~\cite{DIopt1} introduced the notion of consistency for diversity indices: let $\hat{T}=(T,\ell)$ be a rooted phylogenetic $X$-tree and let $\varphi_{\hat{T}}$ be a diversity index of $\hat{T}$. Then, we call $\varphi_{\hat{T}}$ \emph{consistent} if and only if for all $v\in V_{\text{int}}(T)$, $e_1,e_2\in P(v)$ there exists $k\in\mathbb{R}_{\geq 0}$ such that
\begin{align}\label{def::consistency}
\sum_{x\in X[w]}\gamma(x,e_1)=k\cdot\sum_{x\in X[w]}\gamma(x,e_2)~&~&\forall\,(v,w)\in E(T).
\end{align}
We call the constants~$k$ the \emph{consistency constants} of $\varphi_{\hat{T}}$ and condition~\eqref{def::consistency} the \emph{consistency conditions}. Notice that the consistency conditions are redundant for $e_1=e_2$ because the corresponding consistency constant $k=1$ does not \CorrB{constrain} the choice of numbers $\gamma(x,e_1)$. We illustrate the consistency conditions and the dependency of consistency constants $k$ on vertices $v\in V_{\text{int}}(T)$ and edges in $P(v)$ in the following example:
\begin{exmp}\label{ex::consistency}
Consider the rooted $X$-tree $T$ on the left in Figure~\ref{fig::df1} and any edge length function $\ell$ on $E(T)$. Let $\hat{T}=(T,\ell)$ and let $\varphi_{\hat{T}}$ be a consistent diversity index of $\hat{T}$. Then, continuing Example~\ref{ex::def1},
\begin{align*}
\Gamma_T=\left(\begin{array}{cccccccccc}
1 & 0 & 0 & 0 & 0 & 0 & \frac{1}{2} & \gamma(x_1,e_8) & 0 & \gamma(x_1,e_{10})\\
0 & 1 & 0 & 0 & 0 & 0 & \frac{1}{2} & \gamma(x_2,e_8) & 0 & \gamma(x_2,e_{10})\\
0 & 0 & 1 & 0 & 0 & 0 & 0 & \gamma(x_3,e_8) & 0 & \gamma(x_3,e_{10})\\
0 & 0 & 0 & 1 & 0 & 0 & 0 & 0 & \frac{1}{2} & \gamma(x_4,e_{10})\\
0 & 0 & 0 & 0 & 1 & 0 & 0 & 0 & \frac{1}{2} & \gamma(x_5,e_{10})\\
0 & 0 & 0 & 0 & 0 & 1 & 0 & 0 & 0 & 0
\end{array}\right)
\end{align*}
is the matrix of numbers $\gamma(x,e)$ defining $\varphi_{\hat{T}}$. Observe that $\varphi_{\hat{T}}$ is not consistent if we choose $\Gamma_T(e_8)=(1/3,1/3,1/3,0,0,0)$ and $\Gamma_T(e_{10})=(1/8,1/8,1/2,1/8,1/8,0)$ because $$\gamma(x_1,e_8)+\gamma(x_2,e_8)=\frac{2}{3}=\frac{8}{3}\cdot\frac{1}{4}=\frac{8}{3}\cdot\left(\gamma(x_1,e_{10})+\gamma(x_2,e_{10})\right),$$ but $$\gamma(x_3,e_8)=\frac{1}{3}=\frac{2}{3}\cdot\frac{1}{2}=\frac{2}{3}\cdot\gamma(x_3,e_{10}).$$
If instead we choose $\Gamma_T(e_{10})=(1/7,1/7,1/7,2/7,2/7,0)$, then
\begin{align*}
\gamma(x_1,e_8)+\gamma(x_2,e_8)&=\frac{2}{3}=\frac{7}{3}\cdot\frac{2}{7}=\frac{7}{3}\cdot\left(\gamma(x_1,e_{10})+\gamma(x_2,e_{10})\right),\\
\gamma(x_3,e_8)&=\frac{1}{3}=\frac{7}{3}\cdot\frac{1}{7}=\frac{7}{3}\cdot\gamma(x_3,e_{10}),
\end{align*}
i.e., we find consistency constant $k=7/3$ for vertex $v_2$ and edges $e_8,e_{10}$. Furthermore,
\begin{align*}
\gamma(x_i,e_7)&=\frac{1}{2}=\frac{3}{2}\cdot\frac{1}{3}=\frac{3}{2}\cdot\gamma(x_i,e_{8})~&~&\forall\,i\in\{1,2\},\\
\gamma(x_i,e_7)&=\frac{1}{2}=\frac{7}{2}\cdot\frac{1}{7}=\frac{7}{2}\cdot\gamma(x_i,e_{10})~&~&\forall\,i\in\{1,2\},\\
\gamma(x_i,e_9)&=\frac{1}{2}=\frac{7}{4}\cdot\frac{2}{7}=\frac{7}{4}\cdot\gamma(x_i,e_{10})~&~&\forall\,i\in\{4,5\},
\end{align*}
which means $k=3/2$ for $v_1,e_7,e_8$, $k=7/2$ for $v_1,e_7,e_{10}$ and $k=7/4$ for $v_3,e_9,e_{10}$. Hence, $\varphi_{\hat{T}}$ is consistent for the second choice of $\Gamma_T(e_{10})$. Observe that imposing the consistency condition equations reduces the number of free choices for entries of $\Gamma_T$ by 4. This means, together with the convexity condition, we reduce the remaining $8$ free choices of numbers $\gamma(x,e)$ to define $\Gamma_T$ to $2$.
\end{exmp}

We call the number of free choices we can make to define a consistent diversity index of a rooted phylogenetic $X$-tree $\hat{T}=(T,\ell)$ the \emph{degrees of freedom} $d_{\hat{T}}$ of $\hat{T}$. In Example~\ref{ex::consistency} we observed that $d_{\hat{T}}=2$. We generalize this observation in the following proposition by characterizing matrices $\Gamma_T$ for consistent diversity indices. \CorrB{To this end, we consider the equivalence relation on $E(T)$ induced by the neutrality conditions which invoke the isomorphisms $T[v_1]\cong T[v_2]$.} Let $E_1(T),\dots,E_p(T)\subset E(T)$ be the resulting equivalence classes and associate with them the definitions in Table~\ref{tab::symbols1}: we call $E_f$ a \emph{set of independent edges}. \CorrB{For example, for $T'$ in Figure~\ref{fig::df1}, $\{e_1,e_2,e_3,e_5,e_6,e_{11},e_{12},e_{13}\}$, $\{e_4,e_7\}$, $\{e_8,e_9\}$, $\{e_{10}\}$, $\{e_{14}\}$ are the equivalence classes on $E(T')$ and both $\{e_8,e_{14}\}$ and $\{e_9,e_{14}\}$ are sets of independent edges.}

\begin{table*}[!t]
\centering
    \begin{tabular}{ll}
     \toprule
     	symbol & definition\\ \midrule \\ [-1em]
	$E_c$ & a set of representatives $(u,v)\in E(T)$ of the equivalence classes\\ \\ [-1em] 
	& $E_1(T),\dots,E_p(T)$ such that there exists a pendant subtree $H$ of $T$\\ \\ [-1em] 
	& for which $T[v]$ is $H$-balanced.\\ \\ [-1em]
	$E_f$ & a set of representatives $(u,v)\in E(T)$ of the equivalence classes\\ \\ [-1em] 
	& $E_1(T),\dots,E_p(T)$ such that there does not exist a pendant subtree $H$\\ \\ [-1em] 
	& of $T$ for which $T[v]$ is $H$-balanced.\\ \\ [-1em]
	$\mathcal{C}(T)$ & the set of all subsets $E_c$ in $T$\\ \\ [-1em]
	$U_c(T)$ & the union of all edges $e\in E(T)$ with $e\in E_c$ for some subset $E_c$\\ \\ [-1em]
	$\mathcal{F}(T)$ & the set of all subsets $E_f$ in $T$\\ \\ [-1em]
	$U_f(T)$ & the union of all edges $e\in E(T)$ with $e\in E_f$ for some subset $E_f$\\
	\bottomrule \\
    \end{tabular}
    \caption{A collection of definitions associated with the equivalence classes $E_1(T),\dots, E_p(T)$ of a phylogenetic $X$-tree $(T,\ell)$.}
    \label{tab::symbols1}
\end{table*}

\begin{proposition}\label{prop::degfree}
Let $\hat{T}=(T,\ell)$ be a rooted phylogenetic $X$-tree. Then, the degrees of freedom $d_{\hat{T}}$ is the size of a set of independent edges of $T$.
\end{proposition}
\begin{proof}
Let $E_f\in\mathcal{F}(T)$. First, assume there exists $e=(u,v)\in E_f$ such that $\{(u,v)\}=E_q(T)$ for some $q\in\{1,\dots,p\}$. By definition of equivalence class $E_q(T)$, for all edges $e'=(u',v')\in E(T)$ distinct from $(u,v)$ we have \CorrB{$T[v]\ncong T[v']$}. Hence, the neutrality conditions do not induce \CorrB{a dependence} between columns $\Gamma_T(e)$ and $\Gamma_T(e')$. Moreover, we show that for a consistent diversity index $\varphi_{\hat{T}}$ with a matrix of coefficients $\gamma(x,e)$ denoted by $\Gamma_T$, there exists exactly one free choice of entries in $\Gamma_T(e)$. 

To this end, consider $e_1=(v,w_1)$, $e_2=(v,w_2)\in E(T)$. Then, there exist consistency constants $k_1$ and $k_2$ for $w_1,e_1,e$ and $w_2,e_2,e$, respectively: \CorrB{~
\begin{align*}
\sum_{x\in X[w_{1j}]}\gamma(x,e_1)&=k_1\cdot\sum_{x\in X[w_{1j}]}\gamma(x,e)~&~&\forall\,(w_1,w_{1j})\in E(T),\\
\sum_{x\in X[w_{2j}]}\gamma(x,e_2)&=k_2\cdot\sum_{x\in X[w_{2j}]}\gamma(x,e)~&~&\forall\,(w_2,w_{2j})\in E(T).
\end{align*}
Hence,}
\begin{align}
\sum_{x\in X[v]}\gamma(x,e)&=\sum_{i,j\in\{1,2\}}\sum_{x\in X[w_{ij}]}\gamma(x,e)=\sum_{i,j\in\{1,2\}}\frac{1}{k_i}\gamma(x,e_i)\nonumber\\
&=\frac{1}{k_1}\sum_{x\in X[w_1]}\gamma(x,e_1)+\frac{1}{k_2}\sum_{x\in X[w_2]}\gamma(x,e_2).\label{eq::decompG}
\end{align}
The same argument can be extended to all edges in $T[w_1]$ and $T[w_2]$. \CorrB{At some point in this recursion both $T[w_1]$ and $T[w_2]$ are balanced. In this case, for $n_1=|X[w_1]|$ and $n_2=|X[w_2]|$, the height of $T[w_1]$ and $T[w_2]$ are $\log n_1$ and $\log n_2$, respectively. Then, the number of free choices for coefficients $\gamma(x,e)$, $x\in X[v]$, is reduced by $n_i-1$ consistency conditions for each $i\in\{1,2\}$. This means,} there exists exactly one free choice of entries in $\Gamma_T(e)$ independent of $\Gamma_T(e')$ for all $e'$ in $T[w_1]$. The same conclusion can be drawn for $T[w_2]$. \CorrB{Combining both yields $n_1+n_2-2$ consistency conditions reducing the free choices of entries in $\Gamma_T(e)$.} Hence, by the convexity condition there exists \CorrB{at most $n-(n_1+n_2-2)-1=1$} free choice of entries in $\Gamma_T(e)$ independent of $\Gamma_T(e')$ for all $e'$ in $T[v]$. \CorrB{Following the recursive decomposition~\eqref{eq::decompG} back up to the initial choice of edge $e$ this conclusion persists because of the consistency conditions for both $w_1,e_1,e$ and $w_2,e_2,e$. Indeed, by the definition of $E_f$, for $e=(u,v)\in E_f$, no further conditions are imposed on $\gamma(x,e)$, $x\in X[v]$. In other words, exactly one free choice of entries in $\Gamma_T(e)$ is independent of $\Gamma_T(e')$ for all $e'$ in $T[v]$. In total,} we conclude that if all edges in $E_f$ are unique representatives of their respective equivalence classes $E_1(T),\dots,E_p(T)$, then $d_{\hat{T}}=p=|E_f|$. \CorrB{If $|E_q(T)|\geq 2$ for any $q\in\{1,\dots,p\}$, no further neutrality conditions are present for pairs of coefficients appearing in equation~\eqref{eq::decompG}. Thus, our claim follows.}
\end{proof}

We call $E_c\in\mathcal{C}(T)$ a \emph{set of dependent edges} (see Table~\ref{tab::symbols1}). Observe from the proof of Proposition~\ref{prop::degfree} that for all $E_f\in\mathcal{F}(T)$ there exists $E_c\in\mathcal{C}(T)$ such that, for $e=(u,v)\in E_c$, either $T[v]$ is balanced or $\Gamma_T(e)$ is uniquely defined by $\Gamma_T(e')$ for some $e'\in E(T[v])\cap E_f$ because only edges in $U_f(T)$ offer free choices to define $\Gamma_T$ under consistency conditions. The same holds for subsets of~$E_f$ and suitable subsets of~$E_c$. In this case we call (subsets of)~$E_c$ \emph{dependent on} (subsets of)~$E_f$. Throughout this article we will use sets $E_f\in\mathcal{F}(T)$ to define matrices $\Gamma_T$ and only consider sets $E_c\in\mathcal{C}(T)$ dependent on $E_f$ when calculating an allocation of edge lengths to taxa.

Proposition~\ref{prop::degfree} also shows that consistency conditions do not localize all neutrality conditions to their respective edges. \CorrB{In other words, there can exist pairs of edges with a shared neutrality condition which is not affected by the imposition of consistency conditions.} Specifically, \CorrB{the edge length allocation for} edges $(u,v)$ and $(u',v')$ which satisfy a shared neutrality condition but for which $T[v]$ and $T[v']$ do not share any vertices or edges cannot \CorrB{be restricted to either $(u,v)$ or $(u',v')$. For example, edges $e_8$ and $e_9$ in $T'$ in Figure~\ref{fig::df1} share a neutrality condition and appear in consistency conditions when paired with edges from subtrees $T'[v_2]$ or $T'[v_3]$ in the same figure, respectively. However, both $e_8$ or $e_9$ can be chosen to determine $\Gamma_T(e_8)$ and $\Gamma_T(e_9)$ for a consistent diversity index.} Another aspect of describing the neutrality conditions of consistent diversity indices concerns not only localizing them to sets of independent edges but taxa, too. This is of particular interest when we will calculate diversity index scores. To this end, let $\hat{X}_{\hat{x},e}\subseteq X$, $\hat{x}\in \hat{X}_{\hat{x},e}$, $e\in E(T)$, denote the set of maximum cardinality of taxa with constant $\gamma(\hat{x},e)$ induced by neutrality and consistency conditions. 
\begin{exmp}\label{ex::consistent}
Consider the rooted $X$-tree on the left in Figure~\ref{fig::df1}. We have $\hat{X}_{x_1,e}=\{x_1,x_2\}$ for all edges $e\in\{e_7,e_8,e_{10}\}$ because (i) taxa $x_1$ and $x_2$ satisfy a neutrality condition with respect to edge $e_7$; (ii) edges $e_8$ and $e_{10}$ satisfy consistency conditions for $v_1, e_7, e_8$ and $v_1, e_7, e_{10}$, respectively. In this example the dependency of $\hat{X}_{x_1,e}$ on edges $e\in P(x_1)$ is redundant. However, in general this is not the case. For example, on the right in Figure~\ref{fig::df1}, we have $\hat{X}_{x_1,e_7}=\hat{X}_{x_1,e_8}=\{x_1,x_2\}$ and $\hat{X}_{x_1,e_{10}}=\hat{X}_{x_1,e_{14}}=\{x_1,x_2,x_4,x_7\}$ because $e_{10}$ satisfies neutrality conditions 
\begin{align*}
\gamma(x_1,e_{10})=\gamma(x_4,e_{10}),~\gamma(x_2,e_{10})=\gamma(x_7,e_{10})
\end{align*}
or
\begin{align*}
\gamma(x_1,e_{10})=\gamma(x_7,e_{10}),~\gamma(x_2,e_{10})=\gamma(x_4,e_{10})
\end{align*}
and consistency conditions for $v_2,e_8,e_{10}$ and $v_2,e_9,e_{10}$, i.e., 
\begin{align*}
\gamma(x_1,e_{10})=\gamma(x_2,e_{10}),~\gamma(x_4,e_{10})=\gamma(x_7,e_{10}).
\end{align*}
\end{exmp}

\begin{table*}[!t]
\centering
    \begin{tabular}{ll}
     \toprule
     	symbol & definition\\ \midrule \\ [-1em]
	$T(Y)$ & the minimum spanning tree $T[v]$ in $T$ on leafset $Y$ \\ \\ [-1em] 
	& joined with the parent $u$ of $v$ and edge $(u,v)$ if $u$ exists\\ \\ [-1em] 
	$F(Y-Z)$ & the forest we obtain from $T(Y)$ by removing edges \\ \\ [-1em]
	& in paths $P(z)$, $z\in Z$, and deleting resulting singletons\\ \\ [-1em]
	$T_j(Y-x_i)$ & the connected component of $F(Y-\{x_i\})$ which has \\ \\ [-1em]
	& $x_j$ among its leaves\\ \\ [-1em] 
	$E_{f|i}$ & the constriction of $E_f$ to edges in $F(X-\{x_i\})$ \\
	\bottomrule \\
    \end{tabular}
    \caption{A collection of definitions associated with subsets $Z\subseteq Y\subseteq X$ and taxa $x_i,x_j$, $i\neq j$ of a phylogenetic $X$-tree $(T,\ell)$ and $E_f\in\mathcal{F}(T)$.}
    \label{tab::symbols2}
\end{table*}

We will use the set of definitions in Table~\ref{tab::symbols2} to study the allocation of edge lengths to taxa in sets $\hat{X}_{\hat{x},e}$. Specifically, we will decompose $T$ to recursively allocate lengths of edges $e$ to taxa in a set $Y\subseteq X$, $\hat{x}\in Y$, with respect to the neutrality and consistency conditions as encoded by $\hat{X}_{\hat{x},e}$. Considering Table~\ref{tab::symbols2} with a slight abuse of notation we write $F(Y-z)$ whenever $Z=\{z\}$. Observe that connected components $T_j(Y-x_i)$ form an equivalence class in $Y\setminus\{x_i\}$. Then, as a first ingredient for our recursive decomposition approach, for $x_i\in X$, $Y\subseteq X$, we define $\mathcal{R}(x_i,Y)$ as the following set of maximum cardinality:
\begin{enumerate}
\item for each equivalence class $T_j(Y-x_i)$, the set $\mathcal{R}(x_i,Y)$ contains at most one representative,
\item the representatives in $\mathcal{R}(x_i,Y)$ are chosen to pairwise satisfy a maximal number of neutrality and consistency conditions.
\end{enumerate}
Observe that $\mathcal{R}(x_i,Y)$ is not uniquely defined:
\begin{exmp}\label{ex::Ixi}
For the rooted $X'$-tree on the right in Figure~\ref{fig::df1}, we can choose $\mathcal{R}(x_1,X)=\{x_2,x_3,x_j,x_6,x_8\}$ for $j\in\{4,7\}$ to satisfy both properties~1 and~2. \CorrB{Observe that both} $x_4$ and $x_7$ are representatives of $T_5(Y-x_1)$ and share a neutrality condition with $x_2$ for edges $e_8$ and $e_9$. \CorrB{In contrast,} taxon~$x_5$ can not appear in $\mathcal{R}(x_1,X)$ because despite being a representative of equivalence class $T_5(Y-x_1)$, $x_5$ does not share any neutrality condition with the unique representative $x_2$ of $T_2(Y-x_1)$.
\end{exmp}

We will make use of set $\mathcal{R}(x_i,Y)$ at the end of this section when discussing the allocation of edges in $E_{f|i}$ to taxa in $Y\setminus\{x_i\}$ without violating neutrality and consistency conditions. Indeed, if we want to maximize the allocation of edge lengths to taxa in $Y$ under all conditions for diversity indices, starting with an allocation of edge lengths to taxon $x_i$, then property~2 excludes choices of representatives which obstruct maximality. 

As a second ingredient we provide useful formulas for diversity index scores of a single taxon. First, for $Y\subseteq X$, $\hat{x}\in Y$ and a pendant subtree $S$ of $T(Y)$, we call
\begin{align*}
\text{LB}(\hat{x},T)&=\sum_{\substack{e=(u,v)\in P(\hat{x})\cap U_c(T)\\T[v]\text{ balanced}}}\frac{\ell(e)}{\left|\hat{X}_{\hat{x},e}\right|}
\end{align*}
and
\begin{align*}
\text{UB}\left(\hat{x},S\right)&=\,\text{LB}(\hat{x},T)+\sum_{\substack{e=(u,v)\in P(\hat{x})\cap U_c(S)\\T[v]\text{ not balanced}}}\frac{\ell(e)}{\left|\hat{X}_{\hat{x},e}\right|}+\sum_{e\in P(\hat{x})\cap U_f(S)}\frac{\ell(e)}{\left|\hat{X}_{\hat{x},e}\right|}
\end{align*}
the \emph{minimum allocation from $T$} and \emph{maximum allocation from $S$ to $\hat{x}$}, respectively. Notice that the quantity UB$(\hat{x},S)$ contains the term LB$(\hat{x},T)$ which one might expect \CorrB{to depend} on $S$. Here our definition of UB$(\hat{x},S)$ simplifies our notation throughout the article by overlooking evaluation differences between LB$(\hat{x},T)$ and LB$(\hat{x},S)$ which will not affect our structural results.

Next, we parameterize the allocation of lengths for edges $e\in E(S)$ to $\hat{x}$ to move beyond allocations LB$(\hat{x},T)$ and UB$(\hat{x},S)$. At this stage the following definitions are very technical. However, later their high level of detail will simplify our proof of the main result in this article. To this end, we consider the two parameters $E_f\in\mathcal{F}(T(Y))$ and $Z\subseteq Y$ to define $E_f(Z)=\bigcup_{z\in Z}P(z)\cap E_f$ as the set of edges in $E_f$ affecting the allocation of edge lengths to taxa in $Z$ without taking neutrality conditions into account.  Hence, for $e\in E(T)$, the set $E_f\left(\hat{X}_{\hat{x},e}\cap Y\right)$ encodes the edges in $E_f$ affecting the allocation of $\ell(e)$ to $\hat{x}$ in $T(Y)$. Then, for $E^-\subseteq E_f\cap E(S)$, define set
\begin{align*}
\mathcal{E}\left(\hat{x},Y,E_f,E^-,S\right)=&\left(P(\hat{x})\cap E^-\right)\cup\left\{e\in P(\hat{x})\cap E(S)\,:\,\exists\,i\in\{1,\dots,p\},\,e\in E_i(T),\right.\\
&~~~~~~~~~~~~~~~~~~~~~~~~~~~~~~~~~~~~~~~~~~~~\left.\,E_f\left(\hat{X}_{\hat{x},e}\cap Y\right)\cap E_i(T)\neq\emptyset\right\}
\end{align*}
to encode the edges in $P(\hat{x})\cap E(S)$ which either appear in $E^-$ or in the same equivalence class $E_i(T)$, $i\in\{1,\dots,p\}$, as an edge in $E_f\left(\hat{X}_{\hat{x},e}\cap Y\right)$. The purpose of this highly parameterized set is to quantify the dependence of the allocation of a select subset of edge lengths $\{\ell(e)\,:\,e\in E^-\}$ in a pendant subtree of $T(Y)$ to the allocation of edge lengths $\{\ell(e)\,:\,e\in E_f\}$ in $T(Y)$, localized to the set of taxa~$\hat{X}_{\hat{x},e}$. For example, for $E_f'\in\mathcal{F}(S)$ with $E_f'\subset E_f$,
\begin{align*}
\mathcal{E}\left(\hat{x},X,E_f,E_f',S\right)=E_f''\in\mathcal{F}(S)~~~\text{such that}~~~\left|P(\hat{x})\cap E(S)\cap E_f''\right|~~~\text{is maximum.}
\end{align*}
Analogously, we can employ set $\mathcal{E}\left(\hat{x},Y,E_f,E_f',S\right)$ to maximize the allocation of edge lengths $\{\ell(e)\,:\,e\in E_f\cap E(S)\}$ to $\hat{x}$ in $T(Y)$ for any choice of $E_f\in\mathcal{F}(T(Y))$. This formalization will be useful when arguing about maximal allocations of edge lengths to taxa in a recursive setting in the next section. To make such allocations comparable to the quantities LB$(\hat{x},T)$ and UB$(\hat{x},S)$ we extend our definitions a bit further to include sets of dependent edges $E_c\in\mathcal{C}(T(Y))$ that dependent on $\mathcal{E}\left(\hat{x},Y,E_f,E_f',S\right)$: let $\mathcal{E}^*\left(\hat{x},Y,E_f,E^-,S\right)$ denote the union of $\mathcal{E}\left(\hat{x},Y,E_f,E^-,S\right)$ with all sets $E_c'$ such that
\begin{enumerate}
\item $E_c'$ is dependent on $\mathcal{E}\left(\hat{x},Y,E_f,E^-,S\right)$,
\item $\left|P(\hat{x})\cap E_c'\right|$ is maximal in $\mathcal{C}(S)$,
\item $E_c'\subseteq E_c\in\mathcal{C}(S)$ is the smallest subset satisfying (1) and (2).
\end{enumerate}
For example,
\begin{align*}
\mathcal{E}^*\left(\hat{x},X,E_f,E_f',S\right)=P(\hat{x})\cap E(S).
\end{align*}
While the definition of $\mathcal{E}^*\left(\hat{x},Y,E_f,E^-,S\right)$ appears to be overparameterized to describe the allocation of some edge lengths to taxon $\hat{x}$, the free choice of the set of independent edges $E_f$, the set of taxa $Y$ and the pendant subtree $S$ makes our final definition in this section of a maximum allocation from $T$ to a set of taxa $Y$ more concise. We illustrate the definitions introduced after Example~\ref{ex::consistent}:

\begin{figure}[t]
\centering
\includegraphics[scale=0.42]{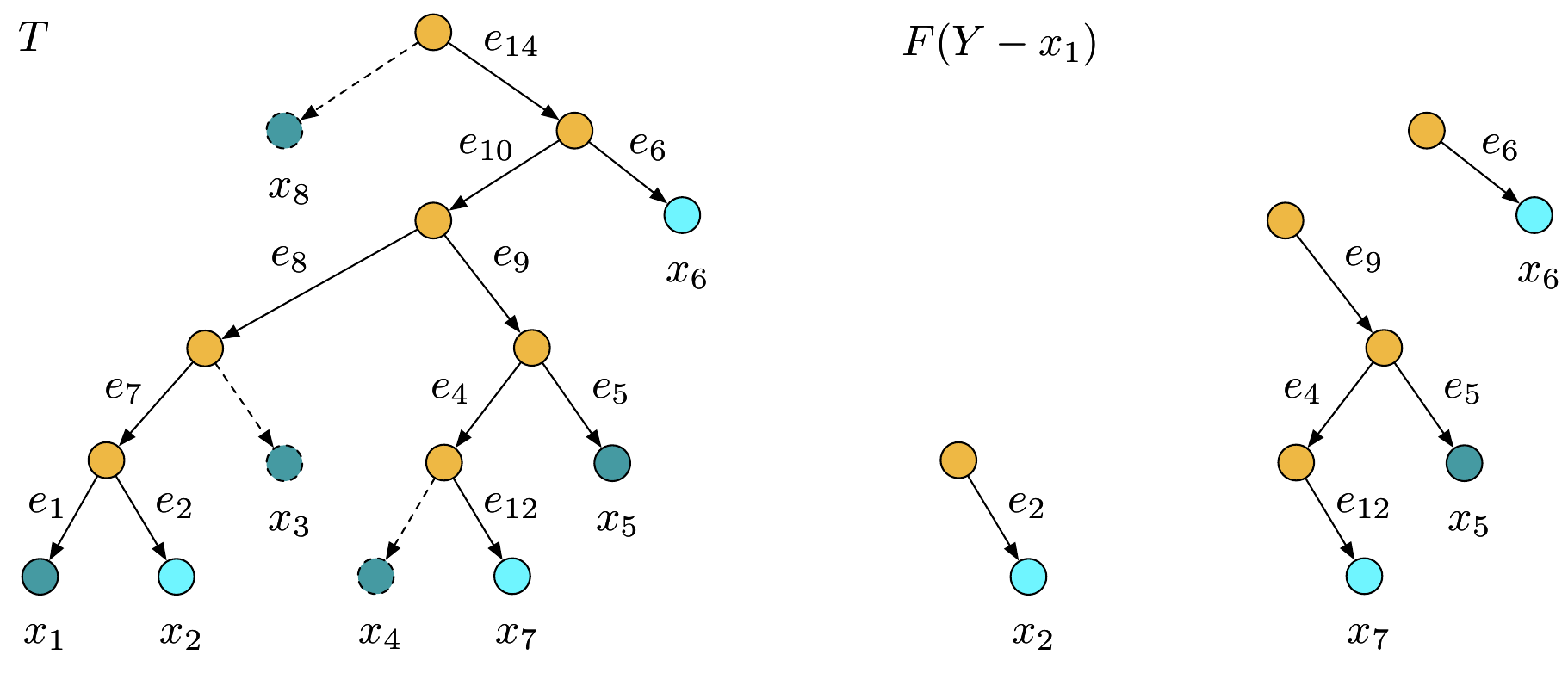}
\caption{On the left a rooted $X$-tree $T$ for $X=\{x_1,x_2,\dots,x_8\}$. For $Y=\{x_1,x_2,x_5,x_6,x_7\}$, the induced subtree $T(Y)$ of $T$ is shown in non-dashed lines. The set $\mathcal{R}(x_1,Y)=\{x_2,x_6,x_7\}$ is shown in light blue. On the right we have the resulting forest $F(Y-x_1)$.}\label{fig::df1.2}
\end{figure}
\begin{exmp}\label{ex::mathcal{E}}
Consider the rooted $X$-tree $T$ on the left in Figure~\ref{fig::df1.2} and the corresponding sets $Y=\{x_1,x_2,x_5,x_6,x_7\}$ and $\mathcal{R}(x_1,Y)=\{x_2,x_6,x_7\}$. Then, for $E_f=\{e_8,e_{14}\}\in\mathcal{F}(T)$ and $E^-=E_f\cap E(T(Y))=E_f$,
\begin{align*}
\mathcal{E}\left(x_1,Y,E_f,E^-,T(Y)\right)=E_f
\end{align*}
because $$\left|P(x_1)\cap E(T(Y))\cap E_f\right|=\left|\{e_1,e_7,e_8,e_{10},e_{14}\}\cap E_f\right|=E_f$$
is maximum among all sets of independent edges in $\mathcal{F}(T(Y))$. This means, since $E_c=\{e_1,e_7,e_{10}\}$ is dependent on $E_f$ with $|P(x_1)\cap E_c|=3$ maximum in $\mathcal{C}(T(Y))$, we have
\begin{align*}
\mathcal{E}^*\left(x_1,Y,E_f,E^-,T(Y)\right)=E_f\cup E_c=\{e_1,e_7,e_8,e_{10},e_{14}\}.
\end{align*}
\end{exmp}
Finally, we are able to describe our recursive decomposition approach to the allocation of edge lengths from $T$. As shorthand notation, let $\hat{X}_{\hat{x}}$ denote the set defining the largest denominator in the definition of LB$(\hat{x},T)$ and let $N(\hat{x})=\left|\hat{X}_{\hat{x}}\right|$. In addition, for $x_i\in Y$, $Z\subseteq Y$, let
\begin{align*}
N(x_i,e,Y)&=\left|\hat{X}_{x_i,e}\cap\bigcup_{y\in Y}\hat{X}_y\right|/\left|\hat{X}_{x_i,e}\right|\in[0,1]~&~&\forall\,e\in E(T).
\end{align*}
Then, we call
\begin{align*}
\text{LB}\left(x_i,Y,T\right)=\sum_{\substack{e=(u,v)\in P(x_i)\cap U_c(T)\\T[v]\text{ balanced}}}N(x_i,e,Y)\cdot \ell(e)=\sum_{\substack{e=(u,v)\in P(x_i)\cap U_c(T)\\T[v]\text{ balanced}}}\ell(e)
\end{align*}
the \emph{minimum allocation from $T(Y)$ to $x_i$}. Here, we have $N(x_i,e,Y)=1$ because $T[v]$ is balanced, i.e., $\hat{X}_{x_i,(u,v)}\subseteq\hat{X}_{x_i}$. Furthermore, for $E^-=E_f\cap E(T(Z))$, we call
\begin{align*}
r_i(T(Z),E_f)=&\sum_{e\in \mathcal{E}^*(x_i,Y,E_f,E^-,T(Z))}N(x_i,e,Y)\cdot\ell(e)\\
&~~~+\sum_{x_j\in \mathcal{R}\left(x_i,Z\right)}r_j\left(T_j(Z-x_i),E_{f|i}\right)
\end{align*}
the \emph{maximum allocation from $T$ to $Y$ with respect to $x_i$, $Z$ and $E_f$}. Intuitively $r_i(T(Z),E_f)$ provides a rough upper bound for the greatest allocation of edge lengths to $x_j$ given that as much as possible is being allocated to $Z$ by leveraging the topological properties of $T(Z)$. The following example illustrates $r_i(T(Z),E_f)$:
\begin{exmp}\label{ex::ri}
Continuing Example~\ref{ex::mathcal{E}}, we consider the rooted $X$-tree $T$ on the left in Figure~\ref{fig::df1.2}, i.e., $Y=\{x_1,x_2,x_5,x_6,x_7\}$, $E^-=E_f=\{e_8,e_{14}\}$, $$\mathcal{R}(x_1,Y)=\{x_2,x_6,x_7\}~~\text{and}~~\mathcal{E}^*\left(x_1,Y,E_f,E^-,T(Y)\right)=\{e_1,e_7,e_8,e_{10},e_{14}\}.$$ 
Since $\bigcup_{y\in Y}\hat{X}_y=Y\cup\{x_4\}$, we get
\begin{align*}
N(x_1,e_1,Y)&=\left|\{x_1\}\cap\left(Y\cup\{x_4\}\right)\right|/\left|\{x_1\}\right|=1\\
\text{and}~~N(x_1,e,Y)&=\left|\{x_1,x_2\}\cap\left(Y\cup\{x_4\}\right)\right|/\left|\{x_1,x_2\}\right|=1~&~&\forall\,e\in\{e_7,e_8,e_{10},e_{14}\}.
\end{align*}
Therefore, $r_1\left(T(Y),E_f\right)$ can be written as
\begin{align*}
&\sum_{e\in\mathcal{E}^*(x_1,Y,E_f,E^-,T(Y))}N(x_1,e,Y)\cdot\ell(e)+\sum_{x_j\in \mathcal{R}\left(x_1,Y\right)}r_j\left(T_j(Y-x_1),E_{f|1}\right)\\
=&\,\text{LB}(x_1,Y,T)+\sum_{e\in\{e_8,e_{10},e_{14}\}}\ell(e)+\sum_{x_j\in\{x_2,x_6,x_7\}}r_j\left(T_j(Y-x_1),E_{f|1}\right).
\end{align*}
Since the constriction $E_{f|1}$ of $E_f$ to edges in $F(Y-x_1)$ is the empty set (see Figure~\ref{fig::df1.2}), for $S=T_2(Y-x_1)$, 
\begin{align*}
r_2\left(S,E_{f|1}\right)=\sum_{e\in\mathcal{E}^*(x_2,Y\setminus\{x_1\},E_f,\emptyset,S)}N(x_i,e,Y)\cdot\ell(e)=\,\text{LB}(x_2,Y,T).
\end{align*}
Analogously, $r_6\left(T_6(Y-x_1),E_{f|1}\right)=\,\text{LB}(x_6,Y,T)$. Similarly, for $S=T_7(Y-x_1)$ and $E^-=E_f\cap E(T(\{x_5,x_7\}))=\emptyset$,
\begin{align*}
\mathcal{E}(x_7,Y\setminus\{x_1\},E_f,E^-,S)=\{e_9\}
\end{align*}
and therefore $\mathcal{E}^*(x_7,Y\setminus\{x_1\},E_f,E^-,S)=\{e_4,e_9,e_{12}\}$. Hence,
\begin{align*}
r_7\left(S,E_{f|1}\right)&=\,\text{LB}(x_7,Y,T)+\ell(e_9)+r_5\left(T'_5(\{x_7\}-x_7),\emptyset\right)\\
&=\,\text{LB}(x_7,Y,T)+\ell(e_9)+\,\text{LB}(x_5,Y,T).
\end{align*}
Thus, in total
\begin{align*}
r_1\left(T(Y),E_f\right)-\sum_{y\in Y}\text{LB}\left(y,Y,T\right)=\ell(e_8)+\ell(e_9)+\ell(e_{10})+\ell(e_{14}).
\end{align*}
Notice that, like with our definition of the quantity UB$(x_i,S)$, the terms LB$(x_i,Y,T)$ appearing in $r_1(T(Z),E_f)$ do not reflect an intuitive understanding of a maximum allocation because for our choice of $Y$ excluding taxon $x_4$ from all calculations yields a tighter upper bound on edge lengths allocated to taxa in $Y$. Again, overlooking this discrepancy will simplify the notation of our further analysis in this article.

Similarly, if we keep $Y=\{x_1,x_2,x_5,x_6,x_7\}$ but consider taxon $x_5$ instead of $x_1$, we arrive at
\begin{align*}
r_5\left(T(Y),E_f\right)-\sum_{x_i\in Y}\text{LB}(x_i,Y,T)=\ell(e_9)+\frac{1}{2}\ell(e_{10})+\frac{1}{2}\ell(e_{14}).
\end{align*}
\end{exmp} 

To conclude this section, we summarize some additional properties of diversity indices investigated by Manson and Steel~\cite{DIopt1} as follows:
\begin{proposition}\label{prop::MS}
Let $\hat{T}=(T,\ell)$ be a rooted phylogenetic $X$-tree.
\begin{enumerate}
\item For every diversity index $\varphi_{\hat{T}}$ which is not consistent there exists a consistent diversity index with the same diversity index scores as $\varphi_{\hat{T}}$.
\item For any diversity index $\varphi_{\hat{T}}$ and $x\in X$, we have $\varphi_{\hat{T}}(x)\in[\text{LB}(x,T),\text{UB}(x,T)]$.
\item $\mathcal{P}(T,\ell)$ is a finite-dimensional convex set.
\item If all diversity indices are consistent, then the dimension of $\mathcal{P}(T,\ell)$ is the maximum number of free choices which specify $\Gamma_T$ for any $\Gamma_Tl\in \mathcal{P}(T,\ell)$.
\end{enumerate}
\end{proposition}
Propositions~\ref{prop::MS}.1 to~\ref{prop::MS}.4 refer to Propositions 11, 5, 7 and Theorem 12 in~\cite{DIopt1}, respectively.

\section{On the combinatorial properties of diversity indices}\label{sec4}
In this section we give a compact description of the diversity index polytope $\mathcal{P}(T,\ell)$ by providing all facet-defining inequalities of an embedding in an appropriate vector space. To this end, we assume that all diversity indices are consistent. \CorrB{This assumption alters the definition of $\mathcal{P}(T,\ell)$ to only contain vectors of diversity index scores which are the image of consistent diversity indices.} In general this assumption does not hold. However, as we will see, consistent diversity indices are sufficient to describe all facets of $\mathcal{P}(T,\ell)$. First, we use Proposition~\ref{prop::MS}.2 to characterize the set of extreme points of $\mathcal{P}(T,\ell)$. The proof is straightforward using the properties of $\mathcal{P}(T,\ell)$ (see Appendix~\ref{proof::STLc1}).
\begin{proposition}\label{prop::STlc1}
Let $\hat{T}=(T,\ell)$ be a rooted phylogenetic $X$-tree and let $E_f\in\mathcal{F}(T)$. Then, the set of extreme points of $\mathcal{P}(T,\ell)$ is given by
\begin{align*}
\mathrm{ext}(T,\ell)=&\left\{\Gamma_Tl\,:\,\text{choose a maximum or minimum value}\right.\\
&~~~~~~~~~\left.\text{for every free choice in }\Gamma_T(e),\,e\in E_f\right\}.
\end{align*}
\end{proposition}
Next, consider the dimension of $\mathcal{P}(T,\ell)$ which we denote by \CorrB{$\dim\mathcal{P}(T,\ell)$}. From Proposition~\ref{prop::MS}.4 we know that \CorrB{$\dim\mathcal{P}(T,\ell)=d_{\hat{T}}$}. This insight enables us to give another characterization of $\mathcal{P}(T,\ell)$ different from Proposition~\ref{prop::STlc1}. Here, we make use of the recursive formula $r_i(T(Z),E_f)$ we derived in the last section to calculate the allocation of edge lengths to different pendant subtrees $T(Z)$ to characterize $\mathcal{P}(T,\ell)$ as an intersection of a finite number of halfspaces.

\subsection{A description of the diversity index polytope by linear inequalities}
First, we define a valid upper bound on the allocation of edge lengths to any subset of taxa: for $Z\subseteq Y\subseteq X$ and $E_f\in\mathcal{F}(T)$, we call
\begin{align}\label{riRecMax}
r\left(T(Z),E_f\right)=\max_{x_i\in Z}~r_i\left(T(Z),E_f\right)-\sum_{z\in Z}\text{LB}\left(z,Y,T\right),
\end{align}
the \emph{maximum allocation for $T(Z)$ with respect to $E_f$}. Example~\ref{ex::ri} has shown that its not clear a-priori which taxa maximize $r_i(T(Z),E_f)$ and subsequent recursion steps because of the dependency of numbers $N(x_i,e,Y)$ on $x_i$ and~$Y$. However, $x_i$ is a maximizer of~\eqref{riRecMax} for a specific choice of $Z$:
\begin{proposition}\label{prop::RecMaxI}
Let $\hat{T}=(T,\ell)$ be a rooted phylogenetic $X$-tree with $d_{\hat{T}}\geq 1$, let $x_i\in X$. Then, there exist $E_f\in\mathcal{F}(T)$ and $Y\subseteq X$ with $|Y|=d_{\hat{T}}$ such that
\begin{align}\label{max::ri}
r\left(T(Y\cup\{x_i\}),E_f\right)=r_i\left(T(Y\cup\{x_i\}),E_f\right)-\sum_{x_j\in Y\cup\{x_i\}}\text{LB}(x_j,Y\cup\{x_i\},T).
\end{align}
Furthermore, there exists a bijective map $\phi :E_f\to Y$ with $e\in P(\phi(e))$ such that 
\begin{align*}
\gamma(\phi(e),e)&=0~&~&\forall\,e\in E_f
\end{align*}
for a diversity index with scores associated to the recursive calculation in~\eqref{max::ri}. Moreover, $z\notin\hat{X}_{y,e}$ for all $y,z\in Y$, $e\in E(T)$.
\end{proposition}
\begin{proof}
Let $E_f\in\mathcal{F}(T)$ and let $\varphi_{\hat{T}}^0$ be the diversity index defined by coefficients $N\left(x_i,(u,v),X\right)$, $T[v]$ not balanced, in the calculation of 
\begin{align}\label{varphi::max}
\max~r_i\left(T,E_f\right)-\sum_{x_k\in X}\text{LB}(x_k,X,T).
\end{align}
Here, $x_i\in X$ is fixed and we maximize all subsequent recursion steps of $r_i\left(T,E_f\right)$. Let $s^0\in \mathcal{P}(T,\ell)$ be the image of $\varphi_{\hat{T}}^0$ and let $T_1',\dots,T_d'$ be the non-singleton connected components of $T$ which we obtain from the recursion $r_i(T,E_f)$ when it terminates in the calculation of~\eqref{varphi::max}. Let $E_f=E_0\cup E_{+}$ be a bipartition of edges $e=(u,v)\in E_f$ with $e_1=(v,w_1),e_2=(v,w_2)\in E(T)$ into edges
\begin{enumerate}
\item $e\in E_0$ such that $e_1\in E\left(T_p'\right)$ for some $p\in\{1,\dots,d\}$,  
\item $e\in E_{+}$ such that $e_i\notin E\left(T_p'\right)$ for all $i\in\{1,2\}$, $p\in\{1,\dots,d\}$. 
\end{enumerate}
Then, define
\begin{align}\label{bij::1}
x^e&=\,\text{arg}\,\min\limits_{x_j\in X[w_1]}\left\{s_j^0-\,\text{LB}(x_j,T)\right\}~&~&\forall\,e=(u,v)\in E_0,(v,w_1)\in E(T).
\end{align}
By definition of $E_0$, taxa $x^e$ are pairwise distinct and $\gamma(x^e,e)=0$, $e\in E_0$, for diversity index $\varphi_{\hat{T}}^0$. Next, let $e=(u,v)\in E_+$. Recall that in the calculation of~\eqref{varphi::max} edges $e_1$ and $e_2$ are removed for different reasons. One edge, say $e_2$, is removed to be consistent with a maximum allocation of edge $e$ to some taxon $x_k$. The other edge $e_1$ is removed for a maximum allocation of $e_1$ to some taxon $x_j\neq x_k$. For the latter, let $P$ denote the directed path starting in $v$ and ending in $x_j$. Suppose by contradiction $\gamma(x_j,e)>0$. Since $\ell(e)$ is (partially) allocated to a taxon $x_k\in X[w_2]$ by $\varphi_{\hat{T}}^0$, we know that $x_j\in\hat{X}_{x_k,e}$. This means, there exists an edge $e'=(u',v')\in P(x_j)\cap P(x_k)$ such that $T[v']$ is $H$-balanced for some graph $H$ and $v',e',e$ satisfy a consistency condition. Here, $e=e'$ is not possible because $e\in E_f$. However, there exists no edge $e'\in P(x_j)\cap P(x_k)$ with $|P(e')|>|P(e)|$ by our choice of $x_j$ and $x_k$, leading to a contradiction. Thus, $\gamma(x_j,e)=0$. Let $x^e=x_j$. Observe that taxa $x^e$, $e\in E_+$, are pairwise distinct because $x_j$ is a singleton component after the removal of the edges in $P$. Hence, for $Y=\{x^e\,:\,e\in E_f\}$, we obtain a bijective map $\phi:E_f\to Y$ with $e\in P(\phi(e))$ and $\gamma(\phi(e),e)=0$ in $\varphi_{\hat{T}}^0$ for all $e\in E_f$. Since $z\notin\hat{X}_{y,e}$ for all $y,z\in Y$, $e\in E(T)$, by our construction of $Y$, diversity index scores of pairs of taxa $x_k,x_l\in Y$ can be dependent only if a neutrality condition $\gamma(x_k,e_k)=\gamma(x_l,e_l)$ holds with either $e_k\in E_f$ or $e_l\in E_f$. Recall that taxa in $Y$ have maximum diversity index scores in~$s^0$ within the connected components they are representatives of in the recursion~$r_i(T,E_f)$. Hence, changing the allocation of lengths of edges in $E_f$ from $\varphi_{\hat{T}}$ to achieve
\begin{align*}
\max~r_j\left(T(Y\cup\{x_i\}),E_f'\right)-\sum_{x_k\in Y\cup\{x_i\}}\text{LB}(x_k,Y\cup\{x_i\},T).
\end{align*}
for some $E_f'\in\mathcal{F}(T)$ and taxon $x_j\in Y$, $j\neq i$, cannot yield a larger value than~\eqref{varphi::max}. Indeed, the neutrality conditions present for taxa $x_k,x_l\in Y$ can only decrease the sum of allocated lengths of edges in $E_f'$ to taxa in $Y$. Here, $E_f'$ establishes the same edge length allocations to taxa in $Y$ as $E_f$ does to taxa in $X$.
\end{proof}

We will use the construction of $Y$ in Proposition~\ref{prop::RecMaxI} for the rest of this section to ensure that we can identify the maximizer of problem~\eqref{riRecMax} without further calculation. Moreover, the bijective map $\phi:E_f\to Y$ indicates that taxa~$Y$ might be sufficient to describe $\mathcal{P}(T,\ell)$ because there exists a diversity index such that the allocation of $\ell(e)$ for any $e\in E_f$ keeps taxon $\phi(e)$ unaffected. In other words, there appear to be $d_{\hat{T}}=|E_f|$ free choices in the calculation of diversity index scores for taxa $Y$. Indeed, we consider $Y\subseteq X$ with $|Y|=d_{\hat{T}}$ and define the \emph{projection of $\mathcal{P}(T,\ell)$ to $Y$} by
\begin{align*}
\text{proj}_{Y}(T,\ell)=\left\{a\in\mathbb{R}_{\geq 0}^{Y}\,:\,a_i=N(x_i)\cdot\left(s_i-\,\text{LB}(x_i,T)\right)~~\forall\,x_i\in Y,~s\in \mathcal{P}(T,\ell)\right\}.
\end{align*}
Then, for $Y$ such that $\text{proj}_{Y}(T,\ell)$ is a proper projection of $\mathcal{P}(T,\ell)$ to $\mathbb{R}^Y$, i.e., the projection preserves the dimension~$d_{\hat{T}}$, we obtain a set of linear constraints which is sufficient to describe $\mathcal{P}(T,\ell)$:
\begin{proposition}\label{prop::STlc2}
Let $\hat{T}=(T,\ell)$ be a rooted phylogenetic $X$-tree with $d_{\hat{T}}\geq 1$. Let $E_f\in\mathcal{F}(T)$, $Y\subseteq X$, $|Y|=d_{\hat{T}}$ and let $\phi:E_f\to Y$ be a bijective map such that the free choice of numbers $\gamma(\phi(e),e)$, $e\in P(\phi(e))$, does fully specify the corresponding matrix $\Gamma_T$. Then,
\begin{align*}
\text{proj}_{Y}(T,\ell)\subseteq\left\{a\in\mathbb{R}_{\geq 0}^{Y}\,:\,\sum_{x_j\in Z}a_j\leq r\left(T(Z),E_f\right)~~~\forall\,Z\subseteq Y\right\}.
\end{align*}
\end{proposition}
\begin{proof}
Let $Y'=\bigcup_{y\in Y}\hat{X}_y$. Since $|Y'|=\sum_{y\in Y}N(y)$ by our choice of $Y$, we have
\begin{align*}
N\left(x_i,e,Y'\right)=\left|\hat{X}_{x_i,e}\cap\bigcup_{y\in Y'}\hat{X}_y\right|/\left|\hat{X}_{x_i,e}\right|=\left|\hat{X}_{x_i,e}\cap\bigcup_{y\in Y}\hat{X}_y\right|/\left|\hat{X}_{x_i,e}\right|=N\left(x_i,e,Y\right).
\end{align*}
Hence, we know that
\begin{align}\label{proj::Yhat}
\sum_{x_j\in Y}N(x_j)\cdot\left(s_j-\,\text{LB}(x_j,T)\right)\leq r\left(T(Y),E_f\right)
\end{align}
is equivalent to
\begin{align}\label{proj::Yprime}
\sum_{x_j\in Y'}s_j-\,\text{LB}(x_j,T)\leq r\left(T(Y'),E_f\right).
\end{align}
Inequality~\eqref{proj::Yprime} holds for all $s\in\mathcal{P}(T,\ell)$ because function $r\left(T(Y'),E_f\right)$ encodes the difference between the maximum and minimum allocation of edge lengths from $T(Y')$ to taxa in $Y'$. By definition of $\phi$, removing one taxon from $Y$ yields an induced subtree $T'$ of $T(Y)$ with $d_{(T',\ell)}=d_{\hat{T}}-1$. Hence, inequality~\eqref{proj::Yhat} holds for all $Z\subseteq Y$, $Z\neq\emptyset$, $s\in\mathcal{P}(T,\ell)$. By definition we can draw the same conclusion for $Z=\emptyset$.
\end{proof}

Proposition~\ref{prop::STlc2} is sufficient to describe $\mathcal{P}(T,\ell)$ as an intersection of halfspaces. However, to prove necessity we embed $\mathcal{P}(T,\ell)$ in a vector space and use a canonical basis therein. 

\subsection{A construction of a basis of diversity indices}
Before we provide a proof for a basis construction for a suitable embedding of $\mathcal{P}(T,\ell)$, we give two preliminary technical results to simplify further arguments in this article: for $x,y\in X$, $E_f\in\mathcal{F}(T)$, $E_c\in\mathcal{C}(T)$ dependent on $E_f$ and weights
\begin{align*}
w_y(x,e)&=\begin{cases}0 &\text{if }y\in\hat{X}_{x,e},\\
1/\left|\hat{X}_{x,e}\right| &\text{otherwise}
\end{cases}~&~&\forall\,e\in E(T),
\end{align*}
we call
\begin{align*}
R\left(x,y,E_f\right)=\sum\limits_{\substack{e=(u,v)\in P(x)\cap P(y)\cap(E_f\cup E_c)\\T[v]\text{ not balanced}}}w_y(x,e)\ell(e)
\end{align*}
the \emph{maximum reallocation from $y$ to $x$ with respect to $E_f$}. Indeed, the quantity $R\left(x,y,E_f\right)$ captures the sum of edge lengths for edges in $P(x)\cap P(y)\cap(E_f\cup E_c)$ which depend on the allocation of edge lengths to $y$ (these edges $e$ neither appear in LB$(x,T)$ nor satisfy $y\in\hat{X}_{x,e}$) normalized by the maximum allocation of these edge lengths to $x$. We are interested in this quantity because the following symmetry relation will be useful when employing parts of our canonical basis to prove which inequalities are facet-defining for $\mathcal{P}(T,\ell)$:
\begin{lemma}\label{lem::Rsym}
Let $T$ by a rooted $X$-tree, $x,y\in X$, $E_f\in\mathcal{F}(T)$ and $E_c\in\mathcal{C}(T)$ dependent on $E_f$. Let $\hat{e}=(u,v)\in P(x)\cap P(y)\cap(E_f\cup E_c)$ with $T[v]$ not balanced and $|P(\hat{e})|$ maximum. Then,
\begin{align*}
\left|\hat{X}_{x,\hat{e}}\right|\cdot R\left(x,y,E_f\right)=\left|\hat{X}_{y,\hat{e}}\right|\cdot R\left(y,x,E_f\right)
\end{align*}
\end{lemma}
\begin{proof}
Let $e=(u',v')\in P(x)\cap P(y)\cap(E_f\cup E_c)$ with $T[v']$ not balanced. Assume $w_y(x,e)=1/\left|\hat{X}_{x,e}\right|$. Then, $w_x(y,e)=1/\left|\hat{X}_{y,e}\right|$. By the consistency conditions, $\left|\hat{X}_{x,\hat{e}}\right|\leq\left|\hat{X}_{x,e}\right|$ and $\left|\hat{X}_{y,\hat{e}}\right|\leq\left|\hat{X}_{y,e}\right|$. Moreover, strict inequality holds only if there exists a vertex $b\in P(u)\setminus P(u')$ which is the root of a $H$-balanced pendant subtree of $T$ for some graph $H$. In this case, for $(a,b)\in E(T)$, $\left|\hat{X}_{x,(a,b)}\right|=2^k\cdot\left|\hat{X}_{x,\hat{e}}\right|$ and $\left|\hat{X}_{y,(a,b)}\right|=2^k\cdot\left|\hat{X}_{y,\hat{e}}\right|$ for some integer $k\geq 1$ because $(a,b)\in P(x)\cap P(y)$. Hence, $\left|\hat{X}_{x,e}\right|/\left|\hat{X}_{x,\hat{e}}\right|=\left|\hat{X}_{y,e}\right|/\left|\hat{X}_{y,\hat{e}}\right|$. Thus, our claim follows.
\end{proof}

Now, to construct a basis for our desired vector space we will use the following technical result (proof in Appendix~\ref{proof::rank}):

\begin{lemma}\label{lem::rank}
For an integer $m\geq 2$, 
\begin{itemize}
\item let $A\in\mathbb{R}^{m\times m}$ such that $a_{ij}=a_{ik}> 0$ for all $i,j,k\in\{1,\dots,m\}$,
\item let $B\in\mathbb{R}^{m\times m}$ such that $b_{11}\geq 0$, $b_{ii}>0$ for $i\in\{2,\dots,m\}$, and $B$ is zero in non-diagonal entries except for at most $m-1$ negative entries and at most one negative entry in both row $i$ and column $i$, $i\in\{1,\dots,m\}$.
\end{itemize}
Then, $A+B$ has full rank.
\end{lemma}

In addition, to simplify our notation of our basis construction and throughout the rest of the article we assume whenever we select a taxon $x_i\in X$ for a rooted phylogenetic $X$-tree $\hat{T}$ that its label $i$ is at most $d_{\hat{T}}+1$.

\begin{proposition}\label{prop::basis}
Let $\hat{T}=(T,\ell)$ be a rooted phylogenetic $X$-tree with $d_{\hat{T}}\geq 1$.
\begin{enumerate}
\item For $x_i\in X$, there exist $E_f\in\mathcal{F}(T)$ and a subset $\left\{s^0,s^1,\dots,s^{d_{\hat{T}}}\right\}\subseteq \mathcal{P}(T,\ell)$ of $d_{\hat{T}}+1$ \CorrB{linearly independent} vectors satisfying
\begin{align*}
s_i^0&=\,\text{UB}\left(x_i,T\right),\\
s_i^j&=\begin{cases}\text{UB}\left(x_i,T\right)-R\left(x_i,x_j,E_f\right) &\text{if }j<i,\\
\text{UB}\left(x_i,T\right)-R\left(x_i,x_{j+1},E_f\right) &\text{otherwise}
\end{cases}~&~&\forall\,j\in\left\{1,\dots,d_{\hat{T}}\right\}.
\end{align*}
In addition, for $j\in\{1,\dots,d_{\hat{T}}+1\},i\neq j$, we have
\begin{align*}
&s_j^0\geq s_j^k,~~~s_k^k>s_j^k~&~&\forall\,k\in\{1,\dots,i-1\},j\neq k,\\
&s_j^0\geq s_j^k,~~~s_{k+1}^k>s_j^k~&~&\forall\,k\in\{i,\dots,d_{\hat{T}}\},j\neq k+1.
\end{align*}
\item \CorrB{$\dim\mathcal{P}(T,\ell)=d_{\hat{T}}$} and the vector space of minimal dimension that includes $\mathcal{P}(T,\ell)$ has dimension $d_{\hat{T}} +1$.
\end{enumerate}
\end{proposition}
\begin{proof}
\underline{1.:} Let $\varphi_{\hat{T}}^0$, $E_f\in\mathcal{F}(T)$ and $Y=\left\{x^e\,:\,e\in E_f\right\}$ with $|Y|=d_{\hat{T}}$ be the diversity index, set of independent edges and set of taxa, respectively, constructed in the proof of Proposition~\ref{prop::RecMaxI}. Let $s^0$ denote the image of $\varphi_{\hat{T}}^0$ and sort $s^0$ in such a way that $s_i^0$ remains the score of $x_i$ and the scores of taxa in $Y$ are among the first $d_{\hat{T}}+1$ entries of $s^0$. Such an ordering of $s^0$ always exists because $|Y|=d_{\hat{T}}$ and $x_i\neq x^e$ for all $e\in E_f$. Relabel all taxa accordingly such that $s_j^0$ is the score of taxon $x_j$ for all $j\in\{1,\dots,n\}$ and recall from Proposition~\ref{prop::RecMaxI} that $\gamma\left(x^e,e\right)=0$ for all $e\in E_f$ in $\varphi_{\hat{T}}^0$. Next, construct a set of $d_{\hat{T}}$ diversity index scores from~$s^0$: for each $j\in\left\{1,\dots,d_{\hat{T}}+1\right\},j\neq i$, increase $\gamma(x_j,e)$ for $x_j=x^e$, $e\in E_f$, to increase score $s_j^0$. This increase is always possible because $\gamma(x^e,e)=0$. Then, entry $s_k^0$ with $\gamma(x_k,e)>0$ decreases as well as entries $s_l^0$ for which $x_k,x_l$ and $e$ share a neutrality condition. Denote the resulting set of diversity index scores by $\left\{s^1,\dots,s^{d_{\hat{T}}}\right\}$. 
Observe that $s^j_i<$\,UB$(x_i,T)$ for $j<i$ and $s^{j-1}_i<$\,UB$(x_i,T)$ for $j>i$ can occur only if $P(x_i)\cap P(x_j)\neq\emptyset$. In this case we ensure that the increase from $s_j^0$ to $s_j^j$ is maximal and therefore 
\begin{align*}
s_i^j&=\begin{cases}\text{UB}\left(x_i,T\right)-R\left(x_i,x_j,E_f\right) &\text{if }j<i,\\
\text{UB}\left(x_i,T\right)-R\left(x_i,x_{j+1},E_f\right) &\text{otherwise}
\end{cases}~&~&\forall\,j\in\left\{1,\dots,d_{\hat{T}}\right\}.
\end{align*}
Since the transformation from $s^0$ to $s^j$ is defined by an edge $e_j=(u_j,v_j)\in E_f$, we can sort and relabel vectors $s^j$, $j\in\left\{1,\dots,d_{\hat{T}}\right\}$, and the corresponding taxa $x_j$ such that $e_l\notin E(T[v_j])$ for all $l\in\left\{j+1,\dots,d_{\hat{T}}\right\}$. Next, consider the matrix $M=\left[s^1-s^0|\dots|s^{d_{\hat{T}}}-s^0\right]$. By definition, each column of $M$ has exactly one positive entry and at least one negative entry among its first $d_{\hat{T}}+1$ rows. If a column $s^j-s^0$ has more than one negative entry, then there exist taxa $x_k,x_l\in X$ with $k,l\in\{i,j+1,\dots,d_{\hat{T}}\}$ such that $x_k,x_l$ and $e_j$ share a neutrality condition. Otherwise an increase in $\gamma(x^j,e_j)$ cannot decrease the score of more than one taxon by the construction of set $Y$. Since $x_j=x^{e_j}$ and $x_i$ is the only choice for $x_k$ or $x_l$ not in $Y$, we need to have $x_k=x^{e_k}$, $e_k\neq e_j$, or $x_l=x^{e_l}$, $e_l\neq e_j$. For $x_k=x^{e_k}$, $k\in\{j+1,\dots,d_{\hat{T}}\}$, we have
\begin{align*}
s_k^j-s_k^0&<0,~&~s_k^k-s_k^0&>0,\\
s_j^j-s_j^0&>0,~&~s_j^k-s_j^0&=0.
\end{align*}
The last equation follows from the fact that $e_k\notin E(T[v_j])$ and $\gamma(x^j,e_j)=0$ for $\varphi_{\hat{T}}^0$. Hence, the vector
\begin{align*}
\left(s^j-s^0\right)-\frac{s_k^0-s_k^j}{s_k^k-s_k^0}\left(s^k-s^0\right)
\end{align*}
has exactly one less negative entry than the pairwise distinct negative entries in $s^j-s^0$ and $s^k-s^0$ combined. We can repeat our arguments for $s^k-s^0$ in place of $s^j-s^0$ to consecutively remove negative entries from the rows of $M$ by linear column transformations. Thus, we conclude there exists a matrix $B$ with the same rank as the submatrix of $M$ consisting of its first $(d_{\hat{T}}+1)$ rows and columns, and exactly one negative and positive entry in each column. Let $A$ denote the submatrix of the first $(d_{\hat{T}}+1)$ rows of matrix $\left[s^0|\dots|s^0\right]$, and let $B_i$ be the matrix we obtain from $B$ by deleting the $i$-th column and append an all-zero column as the first column for all $i\in\{1,\dots,d_{\hat{T}}\}$. Then, from Lemma~\ref{lem::rank} we deduce that matrices $(A+B),(A+B_1),\dots,(A+B_{d_{\hat{T}}})$ all have full rank. This means, vectors $s^0,s^1,\dots,s^{d_{\hat{T}}}$ are linearly independent.\\~\\
\underline{2.:} For $x_i\in X$, we consider $E_f\in\mathcal{F}(T)$ to obtain a subset $B(x_i)\subset \mathcal{P}(T,\ell)$ of $d_{\hat{T}}+1$ linear independent vectors from~(1). Then we have \CorrB{$\dim\mathcal{P}(T,\ell)\geq d_{\hat{T}}$} because $\mathcal{P}(T,\ell)$ is a convex set. We show that $B(x_i)$ is a maximal linear independent set. 

Let $s\in \mathcal{P}(T,\ell)$. Without loss of generality we assume $s\in\text{ext}(T,\ell)$. For $j\in\{1,\dots,d_{\hat{T}}\}$ and some $e\in E_f$, observe that $s^j-s^0$ encodes an increase of $s_{j}^0$ with $x^e=x_j$ and possibly a decrease of some score $s_{k}^0>\,\text{LB}(x_{k},T)$. Recall that $\gamma(x_{j},e)=0$ for diversity index $\varphi_{\hat{T}}^0$. Hence, any choice of a positive value for $\gamma(x_{j},e)$ properly fits our definition of $s^j$. This means, for $e\in E_f$, $j\in\{1,\dots,d_{\hat{T}}\}$, the free choice of $\gamma(x_j,e)$, while fixing the rest of $\Gamma_T$ to yield the values of $s^0$, fully specifies the score $s^j$. Therefore, we conclude that making appropriate free choices for all $\gamma(x^e,e)$, $e\in E_f$, yields $s=\sum_{j=1}^{d_{\hat{T}}}s^j-(d_{\hat{T}}-1)s^0$. This means, $B(x_i)$ is a maximal linearly independent subset of $\mathcal{P}(T,\ell)$. Thus, we conclude that \CorrB{$\dim\mathcal{P}(T,\ell)=d_{\hat{T}}$}.

Now, we have constructed a basis $B(x_i)$ for the affine hull of $\mathcal{P}(T,\ell)$. However, this affine subspace does not include the null vector. Hence, any vector space that includes $\mathcal{P}(T,\ell)$ has to have dimension greater than \CorrB{$\dim\mathcal{P}(T,\ell)$}. Now, taking the linear span of $B(x_i)$ gives us a basis for the desired vector space with dimension \CorrB{$\dim\mathcal{P}(T,\ell)+1$}. Thus, our claim follows.
\end{proof}

\begin{figure*}[!t]
\centering
\includegraphics[scale=0.4]{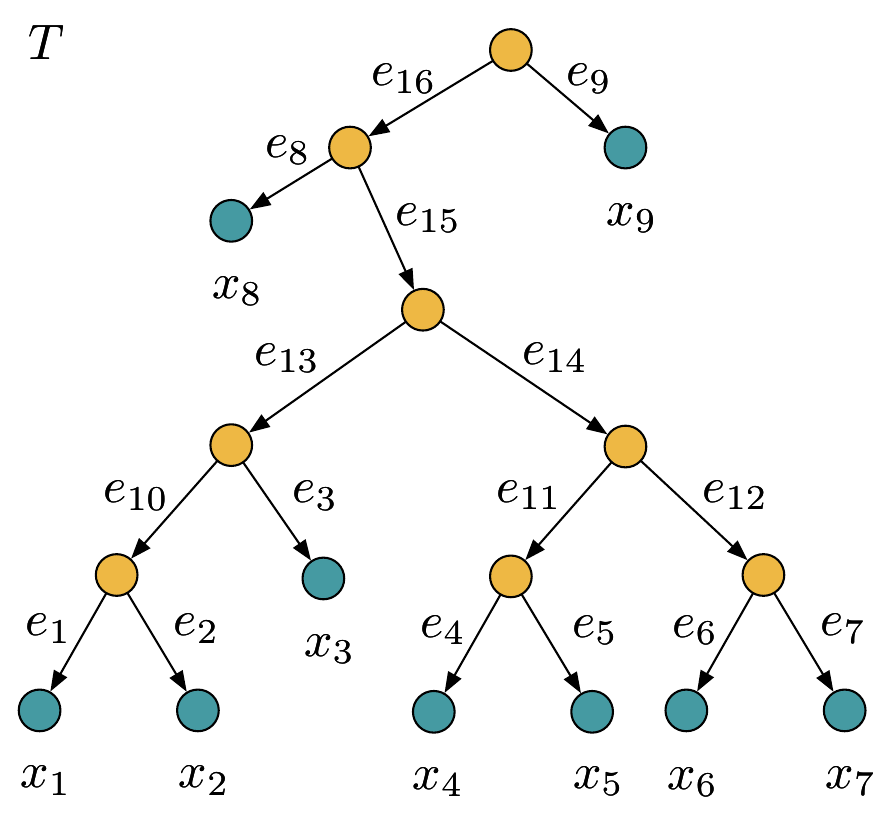}
\caption{A rooted $X$-tree $T$ for $X=\{x_1,x_2,\dots,x_9\}.$}\label{fig::df2}
\end{figure*}

We call any basis $B(x_i)\subset\,\text{ext}(T,\ell)$ of the vector space of minimal dimension that includes $\mathcal{P}(T,\ell)$ which has the same form as constructed in Proposition~\ref{prop::basis}.1 a \emph{canonical basis} and we denote the set of canonical bases by~$\mathcal{B}(T,\ell)$. We illustrate the construction of a canonical basis in the following example:

\begin{exmp}\label{ex::basis1}
Consider the rooted $X$-tree $T$ in Figure~\ref{fig::df2} and any edge length function~$\ell$ on $E(T)$. Let $\varphi_{(T,\ell)}$ be a diversity index. We have $\mathcal{F}(T)=\{\{e_{13},e_{15},e_{16}\}\}$ and therefore $d_{(T,\ell)}=3$. Then, for $x_i=x_1$, Proposition~\ref{prop::RecMaxI} yields the bijection $\phi :\{e_{13},e_{15},e_{16}\}\to\{x_3,x_4,x_8\}$ with $\phi(e_{13})=x_3$, $\phi(e_{15})=x_4$ and $\phi(e_{16})=x_8$. Based on these observations, let $B(x_1)=\{s^0,s^1,s^2,s^3\}\in\mathcal{B}(T,\ell)$. Then, for $i\in\{1,2\}$, $\gamma(x_i,e)=1/2$ with $e\in\{e_{13},e_{15},e_{16}\}$ defines~$s^0$. Furthermore, $\gamma(x_3,e_{13})=1$ defines $s^1-s^0$, $\gamma(x_4,e_{15})=1/4$ defines $s^2-s^0$ and $\gamma(x_8,e_{16})=1$ defines $s^3-s^0$.
\end{exmp}

Notice that in the proof of Proposition~\ref{prop::basis}.1 and Example~\ref{ex::basis1} we have associated a set of independent edges $$E_{B(x_i)}\in\{E_f\,:\,|P(x_i)\cap E_f|\,\text{maximum},\,E_f\in\mathcal{F}(T)\}$$ with each canonical basis $B(x_i)\in\mathcal{B}(T,\ell)$. 
This leads to the following exchange property of canonical bases:

\begin{corollary}\label{cor::canonBasis}
Let $\hat{T}=(T,\ell)$ be a phylogenetic $X$-tree with $d_{\hat{T}}\geq 1$, $x_i\in X$, $B(x_i)=\left\{s^0,s^1,\dots,s^{d_{\hat{T}}}\right\}\in\mathcal{B}(T,\ell)$. Then, there exists $j\in\left\{1,\dots,d_{\hat{T}}+1\right\}$, $j\neq i$, such that interchanging the role of $x_j$ and $x_i$ in $B(x_i)$ yields another basis $B(x_j)=\left\{t^0,t^1,\dots,t^{d_{\hat{T}}}\right\}\in\mathcal{B}(T,\ell)$ with $t_i^0=\,\text{LB}(x_i,T)$.
\end{corollary}
\begin{proof}
Let $(u,v)=\,\text{arg\,max}\left\{|P(e)\cap P(x_i)|\,:\,e\in E_{B(x_i)}\right\}$ and let $x_j=x^{(u,v)}$. By definition of $x^{(u,v)}$ we know that $x_j\neq x_i$. From Proposition~\ref{prop::basis}.1 we obtain a basis $B(x_j)=\left\{t^0,t^1,\dots,t^{d_{\hat{T}}}\right\}\in\mathcal{B}(T,\ell)$. 
By our choice of $x_j$ we deduce that $x_i$ satisfies one of the equations~\eqref{bij::1} for $t^0$. Hence, $x_i$ attains its minimum allocation from $T$ in $t^0$.
\end{proof}

In addition, Example~\ref{ex::basis1} illustrates that the construction of a canonical basis necessitates the existence of the bijection in Propositions~\ref{prop::RecMaxI}. This leads us to combine Propositions~\ref{prop::RecMaxI},~\ref{prop::STlc2} and~\ref{prop::basis} to associate each canonical basis with a valid description of $\mathcal{P}(T,\ell)$:
\begin{corollary}\label{cor::valid}
Let $\hat{T}=(T,\ell)$ be a rooted phylogenetic $X$-tree with $d_{\hat{T}}\geq 1$, let $x_i\in X$, $B(x_i)\in\mathcal{B}(T,\ell)$ and let $Y\subseteq X$ be the taxa associated with the first $d_{\hat{T}}+1$ coordinates in $B(x_i)$. Then, there exists $x_k\in Y$, $k\neq i$, such that
\begin{align*}
\text{proj}_{Y\setminus\{x_k\}}(T,\ell)\subseteq\bigcap_{Z\subseteq Y\setminus\{x_k\},\,x_i\in Z}\mathcal{S}(x_k,Y,Z)
\end{align*}
for
\begin{align*}
\mathcal{S}(x_k,Y,Z)=\left\{a\in\mathbb{R}_{\geq 0}^{Y\setminus\{x_k\}}\,:\,\sum_{x_j\in Z}a_j\leq r_i\left(T(Z),E_{B(x_i)}\right)-\sum_{x_j\in Z}\text{LB}\left(x_j,Y,T\right)\right\}.
\end{align*}
\end{corollary}
\begin{proof}
From Proposition~\ref{prop::basis}.1 we know that sets $E_{B(x_i)}$ and $Y\setminus\{x_i\}$ fit into Proposition~\ref{prop::RecMaxI} as sets $E_f$ and $Y$, respectively. Hence,
\begin{align}\label{eq::rmaxY}
r\left(T(Y),E_{B(x_i)}\right)=r_i\left(T(Y),E_{B(x_i)}\right)-\sum_{x_j\in Y}\text{LB}\left(x_j,Y,T\right).
\end{align}
Moreover, we obtain a bijective map $\phi :E_{B(x_i)}\to Y\setminus\{x_i\}$ with $e\in P(\phi(e))$ for all $e\in E_{B(x_i)}$ from Proposition~\ref{prop::RecMaxI}. First, this means a free choice of numbers $\gamma(\phi(e),e)$ does fully specify the corresponding matrix $\Gamma_T$ of a diversity index. Then, since $|P(x_i)\cap E_{B(x_i)}|$ is maximum in $\mathcal{F}(T)$ by definition of $E_{B(x_i)}$, there exists a bijective map $\phi':E_{B(x_i)}\to Y\setminus\{x_k\}$ for some $x_k\in Y$, $k\neq i$, such that a free choice of numbers $\gamma(\phi'(e),e)$ does fully specify $\Gamma_T$. Therefore, Proposition~\ref{prop::STlc2} yields
\begin{align*}
\text{proj}_{Y\setminus\{x_k\}}(T,\ell)\subseteq\left\{a\in\mathbb{R}_{\geq 0}^{Y\setminus\{x_k\}}\,:\,\sum_{x_j\in Z}a_j\leq r\left(T(Z),E_{B(x_i)}\right)~~~\forall\,Z\subseteq Y\setminus\{x_k\}\right\}.
\end{align*}
Secondly, we can restrict bijection $\phi$ to have a domain and co-domain of size $d_{\hat{T}}-1$ by removing leaf $x_k$ from $T$ and all taxa which share a neutrality condition with respect to edge $\phi^{-1}(x_k)$. Let $T'$ denote the resulting $X'$-tree (up to edge subdivisions) for which $X\cap X'$ is the set of taxa appearing in both $T$ and $T'$. By construction $d_{(T',\ell)}=d_{\hat{T}}-1$, and therefore $(T',\ell)$, $E_{B(x_i)}\setminus\left\{\phi^{-1}(x_k)\right\}$ and $Y\setminus\{x_i,x_k\}$ fit into Proposition~\ref{prop::RecMaxI} as $\hat{T}$, $E_f$ and $Y$, respectively. Since $T'(Y\setminus\{x_k\})=T(Y\setminus\{x_k\})$, we can conclude from further restrictions of bijection $\phi$ that equation~\eqref{eq::rmaxY} holds for all $Z\subseteq Y$ with $x_i\in Z$. Thus, our claim follows.
\end{proof}

We will use Corollary~\ref{cor::valid} to study a proper projection of $\mathcal{P}(T,\ell)$ for which Proposition~\ref{prop::basis} enables us to assess all independence relations between diversity indices.

\subsection{All facets of the diversity index polytope} 
For $Y\subseteq X$ and $x_k\in Y$, let $\mathcal{Z}(x_k,Y)$ denote the subsets $Z\subseteq Y\setminus\{x_k\}$ such that $Y\setminus Z$ are in the same connected component of $F(Y-Z)$. Now, we can provide the facets of $\mathcal{P}(T,\ell)$:
\begin{theorem}\label{thm::facets}
Let $\hat{T}=(T,\ell)$ be a rooted phylogenetic $X$-tree with $d_{\hat{T}}\geq 1$. Then, for $s\in \mathcal{P}(T,\ell)$, 
\begin{enumerate}
\item the inequalities
\begin{align}
s_i&\geq\,\text{LB}(x_i,T)~&~&\forall\,i\in\{1,\dots,n\}\label{facet1}
\end{align}
are facet-defining for $\mathcal{P}(T,\ell)$.
\item $x_i\in X$, $B(x_i)\in\mathcal{B}(T,\ell)$ and $Y\subseteq X$ the taxa associated with the first $d_{\hat{T}}+1$ coordinates in $B(x_i)$, there exists $x_k\in Y$, $k\neq i$, such that the inequalities
\begin{align}
\sum_{x_j\in Z}N(x_j)\cdot\left(s_j-\,\text{LB}\left(x_j,T\right)\right)\leq r\left(T(Z),E_{B(x_i)}\right)\label{facet2}
\end{align}
for all $Z\in\mathcal{Z}(x_k,Y)$ with $x_i\in Z$ are facet-defining for $\mathcal{P}(T,\ell)$.
\end{enumerate}
\end{theorem}
\begin{proof}
\underline{1.:} By definition we know that inequalities~\eqref{facet1} are valid. For $x_i\in X$, let $B(x_i)=\left\{s^0,s^1,\dots,s^{d_{\hat{T}}}\right\}\in\mathcal{B}(T,\ell)$. Then, we know from Corollary~\ref{cor::canonBasis} that there exists a basis $B(x_j)=\left\{t^0,t^1,\dots,t^{d_{\hat{T}}}\right\}\in\mathcal{B}(T,\ell)$ with $t_i^0=\text{LB}(x_i,T)$. By definition of $B(x_j)$ (see Proposition~\ref{prop::basis}.1), we have 
\begin{align*}
t_l^0&\geq t_l^k,~t_k^k>t_l^k~&~&\forall\,l\in\{1,\dots,d_{\hat{T}}+1\},k\in\{1,\dots,j-1\},l\neq k,\\
t_l^0&\geq t_l^k,~t_{k+1}^k>t_l^k~&~&\forall\,l\in\{1,\dots,d_{\hat{T}}+1\},k\in\{j,\dots,d_{\hat{T}}\},l\neq k+1.
\end{align*}
Hence, $\text{LB}(x_i,T)=t_i^0\geq t_i^k$ for all $k\in\{1,\dots,j-1\},i\neq k$ and $k\in\{j,\dots,d_{\hat{T}}\},i\neq k+1$. This means, $x_i$ attains a score of LB$(x_i,T)$ in exactly $d_{\hat{T}}$ elements of $B(x_j)$. In other words, there exists a subset $B_i\subset B(x_j)$ of $d_{\hat{T}}$ linearly independent vectors with a score of LB$(x_i,T)$ for $x_i$ in every element of $B_i$. This means, $B_i$ constitutes a set of $(d_{\hat{T}}-1)$ affinely independent vectors satisfying~\eqref{facet1} with equality. Thus, inequalities~\eqref{facet1} are facet-defining.\\~\\
\underline{2.:} From Corollary~\ref{cor::valid} we know that there exists $x_k\in Y$, $k\neq i$, such that inequalities~\eqref{facet2} are valid for all $Z\subseteq Y\setminus\{x_k\}$ with $x_i\in Z$. Moreover,
\begin{align}
r\left(T(Z),E_{B(x_i)}\right)=r_i\left(T(Z),E_{B(x_i)}\right)-\sum_{x_i\in Z}\text{LB}(x_i,Y,T).\label{part2::r}
\end{align}
We prove our claim by induction on the recursion depth $\delta$ of formula $r_i(T(Z),E_{B(x_i)})$. Let $Z\in\mathcal{Z}(x_k,Y)$ with $x_i\in Z$. Let $Z_j$ denote the leafset of $T_j(Z-x_i)$ and let $E_f=E_{B(x_i)}$.\\

\noindent\underline{Let $\delta =0$:} Then, $\mathcal{R}\left(x_i,Z\right)=\emptyset$. This means, there exist no taxon $x_j$ with $P(x_j)\cap E_f\neq\emptyset$ which is present in $F(Z-x_i)$. This leads to a contradiction because $d_{\hat{T}}\geq 1$ implies $E_f\neq\emptyset$. Therefore, $\delta =0$ cannot occur.\\

\noindent\underline{Let $\delta =1$:} By definition, for $E^-=E_f\cap E(T(Z))$,
\begin{align*}
&r_i\left(T(Z),E_f\right)\\
&=\sum_{e\in \mathcal{E}^*(x_i,Y,E_f,E^-,T(Z))}N(x_i,e,Y)\cdot\ell(e)+\sum_{x_j\in \mathcal{R}\left(x_i,Z\right)}r_j\left(T_j(Z-x_i),E_{f|i}\right).
\end{align*}
Since $\delta =1$, for $x_j\in \mathcal{R}\left(x_i,Z\right)$, we have $\mathcal{R}\left(x_j,Z\right)=\emptyset$. This means, $T_j(Z-x_i)$ does not contain any edges from $E_f$. Hence, $E_{f|i}=\emptyset$ and no paths in $T$ outside $T_j(Z-x_i)$ contain independent edges in the same equivalence class as independent edges in $T_j(Z-x_i)$. Equivalently, for $E_j^-=E_{f|i}\cap E(T(Z_j))$, we have $$\mathcal{E}(x_j,Y,E_f,E_j^-,T(Z_j))=\emptyset.$$ Thus, we conclude
\begin{align*}
r_j\left(T_j(Z-x_i),E_{f|i}\right)&=\,\text{LB}(x_j,Y,T)~&~&\forall\,x_j\in \mathcal{R}\left(x_i,Z\right).
\end{align*}
This means, for $x_j\in Z$ we have $\hat{X}_{x_j,e}=\hat{X}_{x_j}$ for all $e\in E_f$. Therefore, since $x_i\in Y$,
\begin{align*}
N(x_i,e,Y)=\left|\hat{X}_{x_i,e}\cap\bigcup_{y\in Y}\hat{X}_y\right|/\left|\hat{X}_{x_i,e}\right|=1.
\end{align*}
Hence,
\begin{align}
r_i\left(T(Z),E_f\right)&=\sum_{e\in \mathcal{E}^*(x_i,Y,E_f,E^-,T(Z))}\ell(e)+\sum_{x_j\in\mathcal{R}\left(x_i,Z\right)}\text{LB}(x_j,Y,T)\nonumber\\
&=N(x_i)\cdot\left[\text{UB}(x_i,T)-\,\text{LB}(x_i,T)\right]+\sum_{x_j\in Z}\text{LB}(x_j,Y,T).\label{delta1}
\end{align}
Next, we can apply Lemma~\ref{lem::Rsym} for $x_j,x_k\in Z$, $j\neq k$, to get
\begin{align}\label{Rsym}
N(x_j)\cdot R\left(x_j,x_k,E_f\right)=N(x_k)\cdot R\left(x_k,x_j,E_f\right).
\end{align}
In addition, for $B(x_i)=\left\{s^0,s^1,\dots,s^{d_{\hat{T}}}\right\}$ we know from Proposition~\ref{prop::basis}.1 that
\begin{align*}
s_i^0&=\,\text{UB}\left(x_i,T\right),\\
s_i^h&=\,\text{UB}\left(x_i,T\right)-R\left(x_i,x_h,E_f\right)~~\forall\,h\in\{1,\dots,i-1\},\\
s_i^h&=\,\text{UB}\left(x_i,T\right)-R\left(x_i,x_{h+1},E_f\right)~~\forall\,h\in\{i,\dots,d_{\hat{T}}\}.
\end{align*}
Hence, from~\eqref{part2::r} and~\eqref{delta1} we deduce that
\begin{align}
r\left(T(Z),E_f\right)&=N(x_i)\cdot\left[\text{UB}(x_i,T)-\,\text{LB}(x_i,T)\right]\nonumber\\
&=\sum_{x_j\in Z}N(x_j)\cdot\left(s_j^0-\,\text{LB}(x_j,T)\right).\label{r::sh}
\end{align}
Let $h\in\left\{1,\dots,d_{\hat{T}}\right\}$. Without loss of generality $h\in\{1,\dots,i-1\}$. We have
\begin{align*}
N(x_i)\cdot\left(s_i^h-\,\text{LB}(x_i,T)\right)&=N(x_i)\cdot\left(\text{UB}(x_i,T)-\,\text{LB}(x_i,T)-R\left(x_i,x_h,E_f\right)\right)\\
&=N(x_i)\cdot\left(s_i^0-\,\text{LB}(x_i,T)\right)-N(x_i)\cdot R\left(x_i,x_h,E_f\right).
\end{align*}
The same arguments apply when we reverse the role of $i$ and $h$. Hence, using~\eqref{Rsym} we infer that
\begin{align*}
&N(x_i)\cdot\left(s_i^h-\,\text{LB}(x_i,T)\right)+N(x_h)\cdot\left(s_h^h-\,\text{LB}(x_h,T)\right)\\
=\,&N(x_i)\cdot\left(s_i^0-\,\text{LB}(x_i,T)\right)+N(x_h)\cdot\left(s_h^0-\,\text{LB}(x_h,T)\right).
\end{align*}
In addition, for $j\in\{1,\dots,d_{\hat{T}}+1\}$, $i\neq j\neq h$, we have $s_j^h=s_j^0$ because
\begin{align*}
\text{LB}(x_j,T)=s_j^0\geq s_j^h
\end{align*}
follows from Proposition~\ref{prop::basis}.1. Thus,
\begin{align*}
\sum_{x_j\in Z}s_j^0&=\sum_{x_j\in Z}s_j^h~&~&\forall\,h\in\left\{1,\dots,d_{\hat{T}}\right\},~h\neq k.
\end{align*}
In other words, equation~\eqref{r::sh} holds for $s^0$ and $d_{\hat{T}}-1$ pairwise distinct vectors from $\left\{s^h\,:\,h\in\{1,\dots,d_{\hat{T}}\}\right\}$. Thus, $d_{\hat{T}}$ vectors from $B(x_i)$ satisfy inequality~\eqref{facet2} with equality. Moreover, every connected component of $F(Z-x_i)$ contains at most one taxon from $Z$. Therefore, taxa $Y\setminus Z$ are in the same connected component of $F(Y-Z)$, i.e., $Z\in\mathcal{Z}(x_k,Y)$, if and only if $|Z|=d_{\hat{T}}$. Thus, inequality~\eqref{facet2} is facet-defining for $\mathcal{P}(T,\ell)$.\\

\noindent\underline{$\delta\to\delta +1$:} Without loss of generality $\mathcal{R}\left(x_i,Z\right)=\left\{x_1,\dots,x_{d}\right\}$ with $d\leq |Z|-1$ and $x_j=\text{argmax}~r\left(T_j(Z-x_i),E_{f|i}\right)$ for all $j\in\{1,\dots,d\}$. Then, by our induction hypothesis there exist $|Z_j|$ linearly independent vectors $A(x_j)=\left\{s^{j,0},\dots,s^{j,|Z_j|-1}\right\}\subset B(x_j)$ for some basis $B(x_j)=\left\{s^{j,0},\dots,s^{j,d_j}\right\}\in\mathcal{B}\left(T_j(Z-x_i),\ell\right)$, $d_j=d_{(T_j(Z-x_i),\ell)}$, such that 
\begin{align}\label{eq::IHj}
\sum_{x_l\in Z_j}N(x_l)\cdot\left(s_l-\,\text{LB}(x_l,T)\right)&=r\left(T_j(Z-x_i),E_{f|i}\right)~&~&\forall\,s\in A(x_j).
\end{align}
Consider $B(x_i)=\left\{s^0,s^1,\dots,s^{d_{\hat{T}}}\right\}\in\mathcal{B}(T,\ell)$ and denote the edge $\hat{e}$ defined by taxa $x=x_i$ and $y=x_j$ in Lemma~\ref{lem::Rsym} by $\hat{e}(i,j)$. Then, define a set of vectors
\begin{align*}
\mathcal{S}(Z)=\bigcup_{j=1}^d\left\{t^{j,0},\dots,t^{j,|Z_j|-1}\right\}\cup\{s^0\}\subset\mathcal{P}(T,\ell)
\end{align*}
by setting, for $j\in\{1,\dots,d\}$, $t^{j,0}=s^0$ except for 
\begin{align*}
t^{j,0}_l&=s_l^0-R\left(x_i,x_j,E_f\right)~&~&\forall\,x_l\in\hat{X}_{x_i,\hat{e}(i,j)}\cap Z,\\
t^{j,0}_l&=s_l^0+R\left(x_j,x_i,E_f\right)~&~&\forall\,x_l\in\hat{X}_{x_j,\hat{e}(i,j)}\cap Z
\end{align*}
and, for $j\in\{1,\dots,d\}$, $l\in\{1,\dots,|Z_j|-1\}$, $t^{j,l}=s^0$ except for 
\begin{align*}
t_p^{j,l}&=s_p^{j,l}~&~&\forall\,x_p\in Z_j.
\end{align*}
By definition, $s^0$ equals $s^{j,0}$ when constrained to the pendant subtree $T_j(Z-x_i)$. Hence, analogously as for our induction base $\delta =1$ (see equation~\eqref{r::sh}), we know from~\eqref{eq::IHj} and Proposition~\ref{prop::basis}.1 that
\begin{align*}
\sum_{x_j\in Z}N(x_j)\cdot\left(s_j^0-\,\text{LB}(x_j,T)\right)=r_i\left(T(Z),E_f\right)-\,\sum_{x_j\in Z}\text{LB}(x_j,Y,T).
\end{align*}
For $j\in\{1,\dots,d\}$, we have
\begin{align*}
&\sum_{x_l\in Z}N(x_l)\cdot t^{j,0}_l-\sum_{x_l\in Z}N(x_l)\cdot s^0_l\\
&=\sum_{x_l\in\hat{X}_{x_j,\hat{e}(i,j)}\cap Z}N(x_l)\cdot R\left(x_j,x_i,E_f\right)-\sum_{x_l\in\hat{X}_{x_i,\hat{e}(i,j)}\cap Z}N(x_l)\cdot R\left(x_i,x_j,E_f\right)\\
&=\left|\hat{X}_{x_j,\hat{e}(i,j)}\right|\cdot R\left(x_j,x_i,E_f\right)-\left|\hat{X}_{x_i,\hat{e}(i,j)}\right|\cdot R\left(x_i,x_j,E_f\right)=0
\end{align*}
where the second and third equality follow from our choice of $Y$ and Lemma~\ref{lem::Rsym}, respectively. Furthermore, for $l\in\{1,\dots,|Z_j|-1\}$,
\begin{align*}
&\sum_{x_p\in Z}N(x_p)\cdot\left(t^{j,l}_p-\,\text{LB}(x_p,T)\right)\\
&=\sum_{x_p\in Z\setminus Z_j}N(x_p)\cdot\left(s^0_p-\,\text{LB}(x_p,T)\right)+\sum_{x_p\in Z_j}N(x_p)\cdot\left(s^{j,l}_p-\,\text{LB}(x_p,T)\right)\\
&\substack{\eqref{eq::IHj}\\ =}\sum_{x_p\in Z\setminus Z_j}N(x_p)\cdot\left(s^0_p-\,\text{LB}(x_p,T)\right)+r\left(T_j(Z-x_i),E_{f|i}\right)\\
&\stackrel{\text{Proposition}~\ref{prop::basis}.1}{=}r_i\left(T(Z),E_f\right)-\,\sum_{x_j\in Z}\text{LB}(x_j,Y,T).
\end{align*}
Thus, all vectors in $\mathcal{S}(Z)$ satisfy inequality~\eqref{facet2} with equality. Next, we will show that $\mathcal{S}(Z)$ is a set of linearly independent vectors and subsequently show how to extend the set $\mathcal{S}(Z)$ to a set of $d_{\hat{T}}$ linearly independent vectors satisfying inequality~\eqref{facet2} with equality.

For $j\in\{1,\dots,d\}$, observe that the set $B_j=\left\{t^{j,0},\dots,t^{j,|Z_j|-1}\right\}$, when restricted to taxa in $Z_j$, differs from $A(x_j)=\left\{s^{j,0},\dots,s^{j,|Z_j|-1}\right\}$ only in the fact that $t_l^{j,0}=s_l^{j,0}+R(x_j,x_i,E_f)$ for all $x_l\in\hat{X}_{x_j,\hat{e}(i,j)}$. Recall that $s^{j,0}$ is the image of the diversity index defined by the allocation of edges in the calculation of $r(T_j(Z-x_i),E_{f|i})$. This means, $s_l^{j,0}$ is maximum with respect to diversity scores of taxon $x_l\in\hat{X}_{x_j,\hat{e}(i,j)}$ in $\mathcal{P}\left(T_j(Z-x_i),\ell\right)$. Using the maximality of $s_l^{j,0}$ we conclude that the restriction of $B_j$ to $Z_j$ has the same linear independence relations as $A(x_j)$. Moreover, for $h\in\{1,\dots,d\}$, $h\neq j$, $B_h$ restricted to taxa in $Z_j$ yields $|Z_h|$ copies of $s^0$ restricted to $Z_j$. In addition, the restriction of $s^0$ to taxa in $Z_j$ equals $s^{j,0}$ by definition of bases $B(x_i)$ and $B(x_j)$, and $t^{h,0}$ and $t^{j,0}$ are linear independent by construction. This means, $B_h\cup B_j$ is a set of linearly independent vectors. Thus, $\mathcal{S}(Z)\setminus\{s^0\}$ is a set of linearly independent vectors. Furthermore, $s^0$ is linearly independent from vectors $t^{1,0},\dots,t^{d,0}$ by construction. For $j\in\{1,\dots,d\}$, since the restriction of $s^0$ to taxa in $Z_j$ equals $s^{j,0}$, we know that $\{s^0\}\cup B_j\setminus\{t^{j,0}\}$ is a set of linearly independent vectors. Hence, we conclude that $B_j\cup\{s^0\}$ is a set of linearly independent vectors. Thus, in total we conclude that $\mathcal{S}(Z)$ is a set of 
\begin{align*}
1+\sum_{j=1}^{d}|Z_j|=|Z|
\end{align*}
linearly independent vectors. Since $Y\setminus Z$ are in the same connected component of $F(Y-Z)$, for $Z_{d+1}=Y\setminus Z$ and $x_{d+1}\in Z_{d+1}$, there exists a set of $|Z_{d+1}|$ linearly independent vectors $$A(x_{d+1})=\left\{s^{d+1,0},\dots,s^{d+1,|Z_{d+1}|-1}\right\}\subset B(x_{d+1})\in\mathcal{B}\left(T_{d+1}(Z-x_i),\ell\right).$$ For $l\in\{1,\dots,|Z_{d+1}|-1\}$, set $t^{d+1,l}=s^0$ except for
\begin{align*}
t_p^{d+1,l}&=s_p^{d+1,l}~&~&\forall\,x_p\in Y\setminus Z.
\end{align*}
Then, $\mathcal{S}(Z)\cup\left\{t^{d+1,1},\dots,t^{d+1,|Z_{d+1}|-1}\right\}$ is a set of $|Z|+|Y\setminus Z|-1=d_{\hat{T}}$ linearly independent vectors that satisfy inequality~\eqref{facet2} with equality. Thus, our claim follows.
\end{proof}

\begin{figure}[t]
\centering
\includegraphics[scale=0.4]{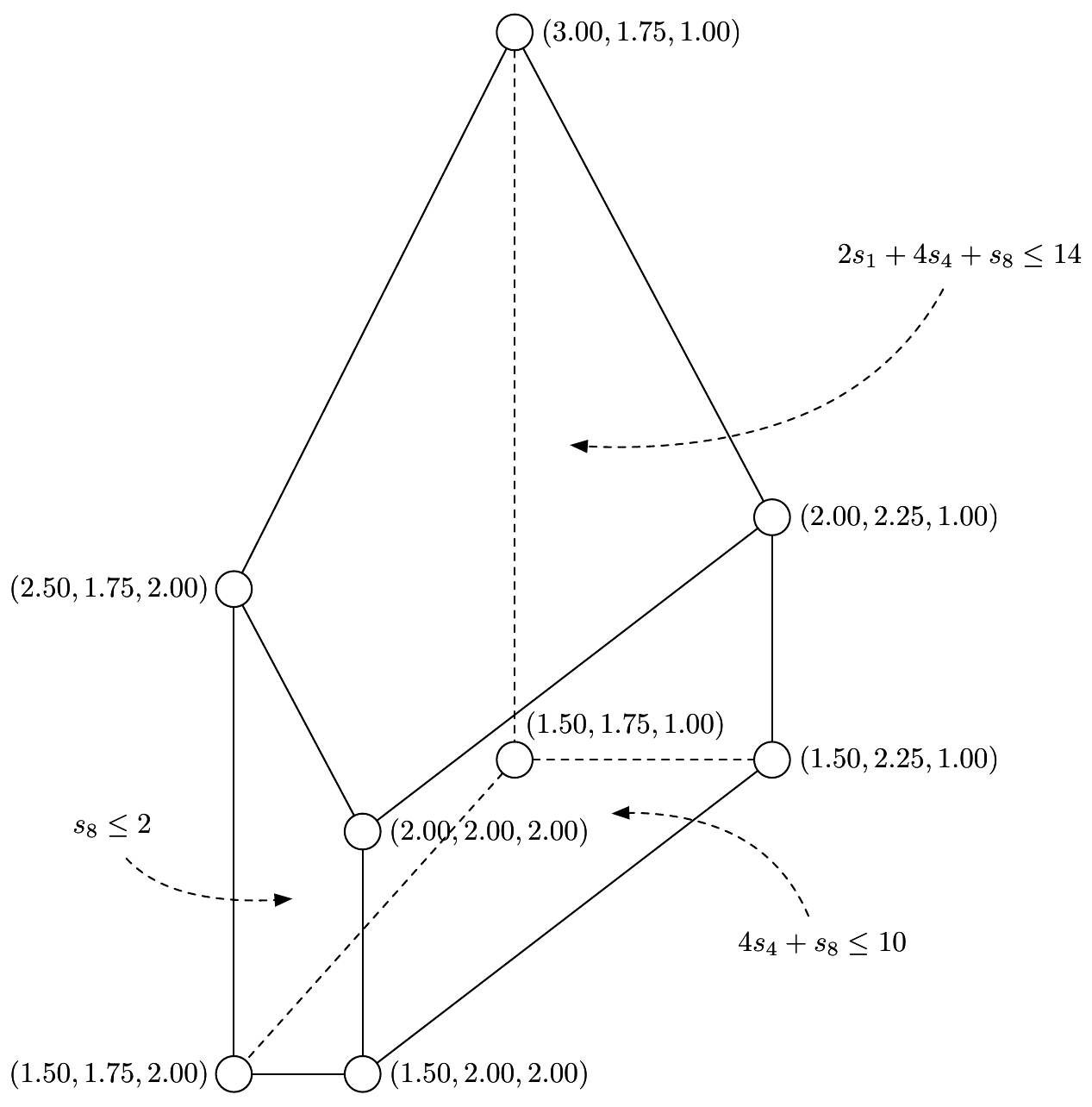}
\caption{The diversity index polytope $\mathcal{P}(T,\ell)$ for the rooted $X$-tree $T$ in Figure~\ref{fig::df2} and edge length function $\ell:E(T)\to\mathbb{R}_{\geq 0}$ with $\ell(e)=1$ such that diversity index scores $s\in\mathcal{P}(T,\ell)$ are projected onto vectors $(s_1,s_4,s_8)$ (using the fact that \CorrB{$\dim\mathcal{P}(T,\ell)=d_{(T,\ell)}=3$}). The polytope $\mathcal{P}(T,\ell)$ can be described by $6$ facets, $3$ of the form~\eqref{facet1} and $3$ of the form~\eqref{facet2}.}\label{fig::3dpoly}
\end{figure}

We illustrate Theorem~\ref{thm::facets} in Figure~\ref{fig::3dpoly}. Specifically, the following example details facets of form~\eqref{facet2}:
\begin{exmp}\label{exmp::facets}
Continuing Example~\ref{ex::basis1}, consider the rooted phylogenetic $X$-tree $(T,\ell)$ in Figure~\ref{fig::df2} and consider $Y=\{x_1,x_3,x_4,x_8\}$ as the first $d_{\hat{T}}+1$ coordinates in $B(x_1)$. Observe that $N(x_3)=N(x_8)=1$, $N(x_1)=2$ and $N(x_4)=4$, and that $x_k=x_3$ is the taxon which fits into Corollary~\ref{cor::valid} (and therefore fits into Theorem~\ref{thm::facets}.2). Let $Z=\{x_1,x_4,x_8\}\in\mathcal{Z}(x_3,Y)$. Then, we have
\begin{align*}
r\left(T(Z),E_f\right)=r_1\left(T(Z),E_f\right)-\sum_{x_j\in Z}\text{LB}\left(x_j,Y,T\right)=\ell(e_{13})+\ell(e_{15})+\ell(e_{16})=3.
\end{align*}
Hence, we know from Theorem~\ref{thm::facets}.2 that
\begin{align*}
\sum_{x_j\in Z}N(x_j)\cdot\left(s_j-\,\text{LB}(x_j,T)\right)=2(s_1-1.5)+4(s_4-1.75)+s_8-1&\leq r\left(T(Z),E_f\right)\\
\Leftrightarrow~~~2s_1+4s_4+s_8&\leq 14
\end{align*}
is a facet of $\mathcal{P}(T,\ell)$. Next, consider $Y$ as the first four coordinates in $B(x_4)$. Then, $x_k\in\{x_1,x_3\}$ fits into Corollary~\ref{cor::valid}. Hence, for $Z=\{x_4,x_8\}\in\mathcal{Z}(x_k,Y)$, we have $$r\left(T(Z),E_f\right)=r_4\left(T(Z),E_f\right)-\sum_{x_j\in Z}\text{LB}\left(x_j,Y,T\right)=\ell(e_{15})+\ell(e_{16})=2$$ and therefore Theorem~\ref{thm::facets}.2 yields facet
\begin{align*}
\sum_{x_j\in Z}N(x_j)\cdot\left(s_j-\,\text{LB}(x_j,T)\right)=4(s_4-1.75)+s_8-1&\leq r\left(T(Z),E_f\right)\\
\Leftrightarrow~~~4s_4+s_8&\leq 10.
\end{align*}
If instead we choose $Z=\{x_1,x_4\}\notin\mathcal{Z}(x_k,Y)$, then Theorem~\ref{thm::facets}.2 is not applicable but we know from Corollary~\ref{cor::valid} that $$r\left(T(Z),E_f\right)=r_4\left(T(Z),E_f\right)-\sum_{x_j\in Z}\text{LB}\left(x_j,Y,T\right)=\ell(e_{13})+\ell(e_{15})+\ell(e_{16})=3$$ and inequality
\begin{align}
\sum_{x_j\in \{x_1,x_4\}}N(x_j)\cdot\left(s_j-\,\text{LB}(x_j,T)\right)=2(s_1-1.5)+4(s_4-1.75)&\leq r\left(T(Z),E_f\right)\nonumber\\
\Leftrightarrow~~~2s_1+4s_4&\leq 13\label{ex::nofacet}
\end{align}
is valid. Indeed, inequality~\eqref{ex::nofacet} is not a facet because $2s_1+4s_4=13$ is the intersection of facets $2s_1+4s_4+s_8=14$ and $s_8=\,\text{LB}(x_8,T)=1$ (see Theorem~\ref{thm::facets}.1). 

Lastly, consider $Y$ as the first four coordinates in $B(x_8)$. Then, $x_k\in Y\setminus\{x_8\}$ fits into Corollary~\ref{cor::valid}. Hence, for $Z=\{x_8\}\in\mathcal{Z}(x_k,Y)$, we have $$r\left(T(Z),E_f\right)=r_8\left(T(Z),E_f\right)-\sum_{x_j\in Z}\text{LB}\left(x_j,Y,T\right)=\ell(e_{16})=1$$ and therefore Theorem~\ref{thm::facets}.2 yields 
\begin{align*}
\sum_{x_j\in Z}N(x_j)\cdot\left(s_j-\,\text{LB}(x_j,T)\right)=s_8-1&\leq r\left(T(Z),E_f\right)~~~\Leftrightarrow~~~s_8\leq 2.
\end{align*}
Figure~\ref{fig::3dpoly} illustrates the three facets from this example. The plane described by equation~\eqref{ex::nofacet} intersects with $\mathcal{P}(T,\ell)$ only in points $\alpha (3,1.75,1)+(1-\alpha)(2,2.25,1)$ for $\alpha\in[0,1]$.
\end{exmp}

We can generalize our observations in Example~\ref{exmp::facets} to prove that Theorem~\ref{thm::facets} gives a characterization of all facets of $\mathcal{P}(T,\ell)$. To this end, let $\mathcal{Z}(T,\ell)$ denote the set of subsets $Z\in\mathcal{Z}(x_k,Y)$, $x_i\in Z$, for which Theorem~\ref{thm::facets}.2 is applicable.

\begin{theorem}\label{thm::compact}
Let $\hat{T}=(T,\ell)$ be a rooted phylogenetic $X$-tree with $d_{\hat{T}}\geq 1$. Then, there exist $n+|\mathcal{Z}(T,\ell)|$ linear inequalities which constitute a minimal compact description of the polytope $\mathcal{P}(T,\ell)$.
\end{theorem}
\begin{proof}
First observe that the union of the descriptions of all projections of $\mathcal{P}(T,\ell)$ to $d_{\hat{T}}$ taxa in Corollary~\ref{cor::valid} contains all facets of $\mathcal{P}(T,\ell)$. Hence, for all $Z\subseteq Y\setminus\{x_k\}$, $x_i\in Z$, we need to prove or disprove that the inequalities in Corollary~\ref{cor::valid} are facet-defining for $\mathcal{P}(T,\ell)$. Theorem~\ref{thm::facets}.1 proves that the non-negativity constraints of the projected diversity index scores are facet-defining. In addition, if $Z\in\mathcal{Z}(x_k,Y)$, $x_i\in Z$, then we obtain facet-defining inequalities~\eqref{facet2} from Theorem~\ref{thm::facets}.2. Otherwise, $Z\notin\mathcal{Z}(x_k,Y)$, $x_i\in Z$, for any choice of $x_k$ which fits into Corollary~\ref{cor::valid}. In this case we show that
\begin{align}\label{cor::com::nofacet}
\sum_{x_j\in Z}N(x_j)\cdot\left(s_j-\,\text{LB}(x_j,T)\right)\leq r\left(T(Z),E_{B(x_i)}\right)
\end{align}
is not facet-defining for $\mathcal{P}(T,\ell)$. Let $W_k\subset Y\setminus Z$ be the taxa which are in the same connected component of $F(Y-Z)$ as $x_k$. Then, $Y\setminus W_k\in\mathcal{Z}(x_k,Y)$. Hence, we know from Theorem~\ref{thm::facets}.2 that inequality
\begin{align}\label{cor::com::facet}
\sum_{x_j\in Y\setminus W_k}N(x_j)\cdot\left(s_j-\,\text{LB}(x_j,T)\right)\leq r\left(T(Y\setminus W_k),E_{B(x_i)}\right)
\end{align}
is facet-defining for $\mathcal{P}(T,\ell)$. Furthermore, taxa $Y\setminus(Z\cup W_k)$ are the taxa in the connected components of $F(Y-Z)$ different from $W_k$. Let $X_1,\dots,X_p$ denote the taxa present in these components, i.e., $\bigcup_{j=1}^pX_j=Y\setminus(Z\cup W_k)$. For $j\in\{1,\dots,p\}$, let $T_j$ denote the connected component in $F(Y-Z)$ with leafset $X_j$. If $\mathcal{F}(T_j)=\emptyset$, then
\begin{align}\label{cor::com::nofacet2}
\sum_{x_l\in X_j}N(x_l)\cdot\left(s_l-\,\text{LB}(x_l,T)\right)\leq r\left(T_j,\emptyset\right)
\end{align}
holds with equality, induced by facets~\eqref{facet1}. Otherwise, we repeat our arguments for $T_j$ instead of $T$. Hence, by induction we conclude that inequality~\eqref{cor::com::nofacet2} is not facet-defining for $\mathcal{P}(T_j,\ell)$. Furthermore, equation~\eqref{cor::com::nofacet} is the intersection of equation~\eqref{cor::com::facet} and equations~\eqref{cor::com::nofacet2} for all $j\in\{1,\dots,p\}$. Thus, inequality~\eqref{cor::com::nofacet} is not facet-defining for $\mathcal{P}(T,\ell)$.
\end{proof}

\section{Conclusions and future research}\label{sec5}
We conclude from Theorem~\ref{thm::compact} that there exists a minimal compact description of the diversity index polytope of polynomial size. So if we characterize a selection criterion for a diversity index by a linear function, then the calculation of the optimal index under this selection criterion can be done in polynomial time by linear programming. 
We envisage two immediate applications 
within this framework, though more are likely. Both contain properties of diversity indices for which precise optimal solutions are unknown, but where finding optima would prove useful to conservation planners. 

The first concerns the extent to which sets of, say, $k$ species with the largest diversity index scores are able to represent the phylogenetic diversity present in the tree as a whole \cite{semple23, redding08}. The separation of these two aspects of phylogenetic diversity could be quantified using a linear difference function $\Delta$ of distinct measures involving diversity indices as our selection criterion. Because such two measures of diversity have different but justifiable underlying assumptions, conservation planners would prefer that they were aligned as much as possible. It is straightforward to calculate the difference $\Delta$ for a given diversity index.
\begin{figure}[t]
\centering
\includegraphics[scale=0.4]{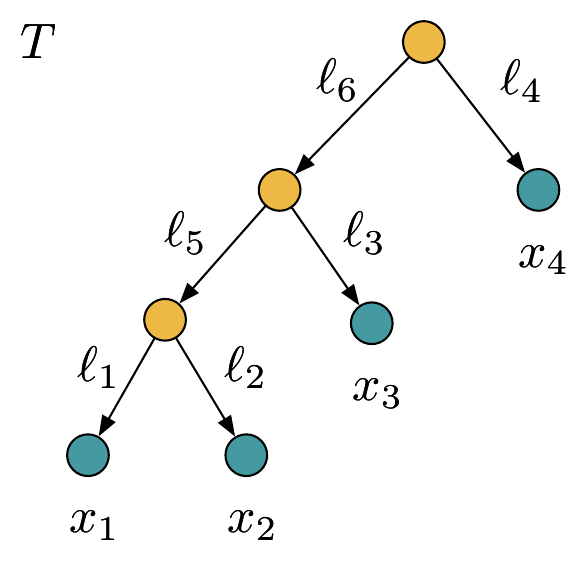}
\caption{A rooted phylogenetic $X$-tree on four leaves with edge lengths $\ell_1,\dots,\ell_5$.}\label{fig::Delta_example}
\end{figure}
For example, for the rooted phylogenetic $X$-tree $(T,\ell)$ in Figure~\ref{fig::Delta_example}, the diversity index scores of the Fair Proportion Index are given by
\begin{align*}
\left(s_1^{\text{FP}},\dots,s_4^{\text{FP}}\right)=\left(\ell_1+\frac{\ell_5}{2}+\frac{\ell_6}{2},\ell_1+\frac{\ell_5}{2}+\frac{\ell_6}{2},\ell(\ell_3)+\frac{\ell_6}{3},\ell_4\right).
\end{align*}
Then, for $x_i,x_j\in X=\{x_1,\dots,x_4\}$, 
\begin{align*}
\Delta(x_i,x_j)=\,\text{phylogenetic diversity of $\{x_i,x_j\}$ }-s_i^{\text{FP}}-s_j^{\text{FP}}
\end{align*}
could be a suitable choice for the linear difference function $\Delta$ on $k=2$ species. Subsequently, we can find $k$ species on a given tree that minimize $\Delta$. However, it is not immediately clear how one would choose a diversity index to minimize $\Delta$ without first knowing which $k$ species will be selected. Taking the perspective of robust optimization~\cite{ben09}, and using the diversity index polytope as a region of uncertainty, we can solve this problem while remaining agnostic to the selected set of $k$ species. In our small example the robust version of $\min\{\Delta(x_i,x_j)\,:\,x_i,x_j\in X\}$ would be
\begin{align*}
\max_{x_i,x_j\in X}\min_{s\in\mathcal{P}(T,\ell)}\,\text{phylogenetic diversity of $\{x_i,x_j\}$ }-s_i-s_j.
\end{align*} 

The second application involves the way that extinctions of species impact diversity index scores of survivors. The demise of a species results in the removal of that leaf and incident edge from the phylogenetic tree. For example, modelled by a stochastic birth-death process~\cite{mooers12}. In doing so, the diversity index scores of remaining leaves are recalculated to reflect the new tree structure. However, this recalculation can lead to quite drastic changes in index scores. Of particular concern, a ranking based on these scores can sometimes become completely reversed after certain extinctions, causing great upset to conservation priorities~\cite{gumbs18}. Finding optimal diversity indices that best preserve rankings in the presence of extinction events would be a key benefit of the present construction. This problem requires further structure that must be added to the diversity index polytope to capture the dependencies between species introduced by a ranking. In this context, approaches like the rotation method introduced by Bolotashvili and Kovalev~\cite{bolotashvili88} to study rankings in the linear ordering polytope~\cite{bolotashvili99} can inform the study of diversity rankings.  

Overall, the results in this articles offer a way of rigorously codifying $\mathcal{P}(T,\ell)$. This opens the door for an analysis of the properties of diversity indices on a broad scale unlike what has previously been attained. In other words, our results introduce the methodological advantages of polyhedral combinatorics to solve problems in conversation biology.

\section{Acknowledgements}
We thank the anonymous reviewers for their detailed feedback which helped to improve the article.
 
\appendix
\section{Proof of Proposition~\ref{prop::STlc1}}\label{proof::STLc1}~\\
Statement of Proposition~\ref{prop::STlc1}: let $\hat{T}=(T,\ell)$ be a rooted phylogenetic $X$-tree and let $E_f\in\mathcal{F}(T)$. Then, the set of extreme points of $\mathcal{P}(T,\ell)$ is given by
\begin{align*}
\mathrm{ext}(T,\ell)=&\left\{\Gamma_Tl\,:\,\text{choose a maximum or minimum value}\right.\\
&~~~~~~~~~\left.\text{for every free choice in }\Gamma_T(e),\,e\in E_f\right\}.
\end{align*}
\begin{proof}
By definition, ext$(T,\ell)\subseteq \mathcal{P}(T,\ell)$. If $d_{\hat{T}}=0$, then $|\mathcal{P}(T,\ell)|=1$ and therefore our claim holds. Hence, assume $d_{\hat{T}}\geq 1$. Let $s=\Gamma_Tl\in\text{ext}(T,\ell)$. Suppose, with a view to contradiction, that $s$ is not an extreme point of $\mathcal{P}(T,\ell)$. Since $\mathcal{P}(T,\ell)$ is convex by Proposition~\ref{prop::MS}.2, there exist $\Gamma_T^jl\in \mathcal{P}(T,\ell)$, $\Gamma_T^j\neq\Gamma_T$, $\lambda_j>0$, $j\in\{1,\dots,k\}$, $k\geq 2$, such that $$\Gamma_Tl=\sum_{j=1}^k\lambda_j\Gamma_T^jl,~~~\sum_{j=1}^k\lambda_j=1.$$ Hence,
\begin{align*}
\gamma(x_i,e)&=\sum_{j=1}^k\lambda_j\gamma^j(x_i,e)~&~&\forall\,i\in\{1,\dots,n\}.
\end{align*}
For $e\in E_f$, there exists a taxon $x\in X$ such that $\gamma(x,e)=0$. To see this, observe that the minimum value for any free choice in $\Gamma_T(e)$ is zero and the maximum value for any free choice of $\gamma(y,e)$, $y\in X$, does imply $\gamma(z,e)=0$ for at least one $z\in X$, $z\neq y$. This means $\sum_{j=1}^k\lambda_j\gamma^j(x,e)=0$. Since $\lambda_j>0$ for all $j\in\{1,\dots,n\}$, we obtain $\gamma^j(x,e)=0$. Therefore, for $j\in\{1,\dots,n\}$, $\Gamma_T(e)=\Gamma_T^j(e)$. However, for $j\in\{1,\dots,n\}$, we assumed $\Gamma_T\neq\Gamma_T^j$, leading to a contradiction. Thus, $s$ is an extreme point of $\mathcal{P}(T,\ell)$.

Conversely, let $s=\Gamma_Tl\in \mathcal{P}(T,\ell)\setminus\text{ext}(T,\ell)$. This means, for at least one edge $e\in E_f$, no free choice in $\Gamma_T(e)$ is maximum or minimum. Let $\Gamma_T^1l,\Gamma_T^2l\in \mathcal{P}(T,\ell)$ such that, for some fixed $x\in X$, $\gamma^1(x,e)$ and $\gamma^2(x,e)$ are given by a maximum and minimum free choice, respectively, and $\Gamma_T^1(e')=\Gamma_T^2(e')=\Gamma_T(e')$ for all $e'\in E(T)\setminus\{e\}$. Hence, for some $\alpha >0$, $\Gamma_T(e)=\alpha\Gamma_T^1(e)+(1-\alpha)\Gamma_T^2(e)$, meaning $s$ is a convex combination of two diversity index scores which are pairwise distinct from $s$. Thus, $s$ is not an extreme point of $\mathcal{P}(T,\ell)$ because $\mathcal{P}(T,\ell)$ is convex.
\end{proof}

\section{Proof of Lemma~\ref{lem::rank}}\label{proof::rank}~\\
Statement of Lemma~\ref{lem::rank}: for an integer $m\geq 2$, 
\begin{itemize}
\item let $A\in\mathbb{R}^{m\times m}$ such that $a_{ij}=a_{ik}> 0$ for all $i,j,k\in\{1,\dots,m\}$,
\item let $B\in\mathbb{R}^{m\times m}$ such that $b_{11}\geq 0$, $b_{ii}>0$ for $i\in\{2,\dots,m\}$, and $B$ is zero in non-diagonal entries except for at most $m-1$ negative entries and at most one negative entry in both row $i$ and column $i$, $i\in\{1,\dots,m\}$.
\end{itemize}
Then, $A+B$ has full rank.
\begin{proof}~
\begin{description}
\item[Case 1:] $B$ is invertible. Since we can write $A=(a_{11},\dots,a_{mm})(1,\dots,1)^T$, we get
\begin{align*}
\det(A+B)=\det(B)\cdot(1+(1,\dots,1)^TB^{-1}(a_{11},\dots,a_{mm}))
\end{align*}
from the matrix determinant lemma. Hence, we can prove $\det(A+B)\neq 0$ by showing $B^{-1}(a_{11},\dots,a_{mm})\in\mathbb{R}_{>0}^m$.  Suppose by contradiction that the unique solution $x\in\mathbb{R}^m$ to the system $Bx=(a_{11},\dots,a_{mm})$ is non-positive. Observe that rows $i$ of $B$ of the form $b_{ii}x_i=a_{ii}>0$ cannot occur for $x_i\leq 0$. Hence, there exists a negative non-diagonal entry in each row of $B$. However, $B$ contains at most $m-1$ negative entries, leading to a contradiction. Therefore, all entries of vector $B^{-1}(a_{11},\dots,a_{mm})$ are positive. Thus, $A+B$ has full rank.
\item[Case 2:] $B$ is singular. Then, rank$(B)\leq m-1$. We will show that $\dim(\ker(B))\leq 1$.\linebreak Then, by the rank-nullity theorem rank$(B)=m-1$ follows. To this end, consider $x\in\ker(B)$. Without loss of generality $b_{11}>0$ and we write $B=D-C$ for a diagonal matrix~$D$ with non-negative entries and a sparse matrix~$C$ with at most one non-zero entry per row and column and at most $m-1$ non-zero entries in total. Then, $Dx=Cx$ or equivalently, for all $i\in\{1,\dots,m\}$,
\begin{align}\label{def::x}
x_{i}=\begin{cases}
\frac{C_{ij}}{D_{ii}}x_j &\text{if there exists $j\in\{1,\dots,m\}$ such that $C_{ij}\neq 0$},\\
0 &\text{otherwise.}
\end{cases}
\end{align}
Let $i_1,\dots,i_{k}$ denote all indices $i\in\{1,\dots,m\}$ for which there exists $j\in\{1,\dots,m\}$ such that $C_{ij}\neq 0$ and let $j_1,\dots,j_k$ denote the corresponding indices $j$. Since $C$ has the same non-zero-entry-pattern as a permutation matrix except for some all-zero rows and columns, it follows that $x\neq 0$ is possible only if
\begin{align*}
\prod_{i_h\in\{i_1,\dots,i_k\}}\frac{C_{i_hj_h}}{D_{i_hi_h}}=1.
\end{align*}
This means, $x\neq 0$ is unique if it exists, i.e., $\dim(\ker(B))\leq 1$. We conclude that rank$(B)=m-1$. \CorrB{In addition, we observe that $x$ can be written as $(x_1,c_2x_1,\dots,c_mx_1)$ for non-negative scalars $c_i$, $i\in\{2,\dots,m\}$, which are defined by~\eqref{def::x}. Since $x\in\ker(B)$ and $B$ is singular, we know that $x\neq 0$. Hence, it follows that $x_1\neq 0$. Therefore, $(A+B)x=Ax=(a_{11},a_{21},\dots,a_{m1})\cdot\sum_{i=1}^mx_i\neq 0$. Thus, $\ker(A+B)=\{0\}$, i.e., $A+B$ has full rank.}
\end{description}
\end{proof}

\end{document}